\documentclass[11pt]{amsart}   

\usepackage{verbatim}
\usepackage{amssymb,amsthm,amsmath}
\usepackage{color}
\definecolor{darkblue}{rgb}{0, 0, .4}
\definecolor{grey}{rgb}{.7, .7, .7}

\usepackage{ifpdf}
\ifpdf
  \usepackage[pdftex]{epsfig}

  \usepackage{lscape}

  \usepackage[breaklinks]{hyperref}
  \hypersetup{
            colorlinks=true,
            linkcolor=black,
            anchorcolor=black,
            citecolor=black,
            pagecolor=black,
            urlcolor=black,
            pdftitle={},
            pdfauthor={}
  }
\else
  \usepackage[dvips]{epsfig}
  \usepackage{lscape}

  \newcommand{\href}[2]{#2}
  \newcommand{\url}[2]{#2}
\fi

\newtheorem{firstdef}{Definition}  
\newtheorem{theorem}{Theorem}[section]
\newtheorem{lemma}[theorem]{Lemma}

\theoremstyle{definition}
\newtheorem{definition}[theorem]{Definition}
\newtheorem{example}[theorem]{Example}

\theoremstyle{remark}
\newtheorem{remark}[theorem]{Remark}

\numberwithin{equation}{section}

\theoremstyle{theorem}
\newtheorem{corollary}[theorem]{Corollary}

\newtheorem{proposition}[theorem]{Proposition}
\newtheorem{conjecture}[theorem]{Conjecture}

\newcommand{\Z}[0]{\mathbb{Z}}

\newcommand{\Q}[0]{\mathbb{Q}}

\newcommand{\ul}[1]{\underline{#1}}

\usepackage{graphicx}
\input xy 
\xyoption{all} 
\newcommand\scalemath[2]{\scalebox{#1}{\mbox{\ensuremath{\displaystyle #2}}}}

\usepackage{tikz}
\usetikzlibrary{decorations.pathreplacing,calc,positioning,trees,matrix}
\usepackage{ifthen}

\DeclareMathOperator{\rowz}{row_{\mathbb{Z}}}
\DeclareMathOperator{\rowq}{row_{\mathbb{Q}}}

\renewcommand{\L}{\mathcal{L}}
\renewcommand{\P}{\mathcal{P}}
\newcommand{\M}{\mathcal{M}}
\newcommand{\E}{\mathcal{E}}
\newcommand{\I}{\mathcal{I}}
\newcommand{\W}{\mathcal{W}}
\newcommand{\U}{\mathcal{U}}
\newcommand{\A}{\mathcal{A}}
\newcommand{\GM}{\mathsf{GM}}
\newcommand{\kr}{{k}_{\mathsf{row}}}
\newcommand{\kc}{{k}_{\mathsf{col}}}
\newcommand{\kk}{{k}_{\mathsf{*}}}
\newcommand{\G}[0]{\Gamma}

\DeclareMathOperator{\diag}{diag}

\title{Integer diagonal forms for subset intersection relations}

\author[J. Ducey]{Joshua E. Ducey}
\address{James Madison University\\ 
Department of Mathematics and Statistics\\
60 Bluestone Dr\\
Roop Hall room 305\\
Harrisonburg VA 22807 (USA)}
\email{\href{mailto:duceyje@jmu.edu}{\texttt{duceyje@jmu.edu}}}

\author[L. Engelthaler]{Lauren Engelthaler}
\address{University of Dallas}
\email{lengelthaler@udallas.edu}

\author[J. Gathje]{Jacob Gathje}
\address{College of Saint Benedict and Saint John's University}
\email{JGATHJE001@csbsju.edu}

\author[B. Jones]{Brant Jones}
\address{James Madison University}
\email{jones3bc@jmu.edu}

\author[I. Pfaff]{Isabel Pfaff}
\address{Oberlin College}
\email{ipfaff@oberlin.edu}

\author[J. Plute]{Jenna Plute}
\address{Texas A\&M University}
\email{jplute@tamu.edu}


\begin{document}
\begin{abstract}
For integers $0 \leq \ell \leq k_{r} \leq k_{c} \leq n$, we give a description for the Smith group of the incidence matrix with rows (columns) indexed by the size $k_r$ ($k_c$, respectively) subsets of an $n$-element set, where incidence means intersection in a set of size $\ell$.  This generalizes work of Wilson and Bier from the 1990s which dealt only with the case where incidence meant inclusion.  Our approach also describes the Smith group of any matrix in the $\Z$-linear span of these matrices so includes all integer matrices in the Bose-Mesner algebra of the Johnson association scheme: for example, the association matrices themselves as well as the Laplacian, signless Laplacian, Seidel adjacency matrix, etc. of the associated  graphs.  In particular, we describe the critical (also known as sandpile) groups of these graphs.  The complexity of our formula grows with the parameters $k_{r}$ and $k_{c}$, but is independent of $n$ and $\ell$, which often leads to an efficient algorithm for computing these groups.  We illustrate our techniques to give diagonal forms of matrices attached to the Kneser and Johnson graphs for subsets of size $3$, whose invariants have never before been described, and recover results from a variety of papers in the literature in a unified way.
\end{abstract}

\maketitle

\section{Introduction}

This paper concerns abelian group invariants for a large class of integer matrices.  Given an $m \times n$ integer matrix $M$, one can view the rows as describing zero relations among the generators of a finitely generated abelian group, $S(M)$, known as the {\bf Smith group} of $M$.  To be explicit,
\[
S(M) = \Z^{n} / \rowz(M),
\]
where $\rowz(M)$ denotes all integer-coefficient linear combinations of the rows of $M$.  (In this paper our vectors will be row vectors, to which we will apply our matrices from the right.)  In the case that $M$ is an incidence matrix, this abelian group is really an invariant of the incidence relation, because its isomorphism type is unchanged by any reordering of the rows or columns of $M$.  

It is well-known that a  finitely generated abelian group can be described as a direct sum of cyclic groups.  Such a description of $S(M)$ is equivalent to finding a {\bf diagonal form} of $M$, by which we mean a matrix $D$ of the same dimensions as $M$ (so ``diagonal'' simply means that the $(i,j)$-entry of $D$ is $0$ unless $i=j$) satisfying 
\[ E M F = D, \]
where $E$ and $F$ are unimodular matrices.  Here, a {\bf unimodular} matrix is a square integer matrix that has an integer inverse and in the situation where two integer matrices are related by unimodular transformations, as above, we say the matrices are {\bf integrally equivalent}.  If $r$ is the rank of $M$ (and so of $D$), then we can write
\[
D = \diag(d_{1}, d_{2}, \cdots, d_{r}, 0, \cdots, 0)
\]
and we have 
\[
S(M) \ \cong\  \Z/d_{1}\Z \ \oplus\ \Z/d_{2}\Z \ \oplus\ \cdots \ \oplus\ \Z/d_{r}\Z \ \oplus\ \Z^{n - r}.
\]
It turns out there is exactly one diagonal form of $M$ that has nonnegative diagonal entries, each dividing the next.  This is called the \textbf{Smith normal form} of $M$ and corresponds to the invariant factor decomposition of the Smith group.

One incidence matrix whose Smith group has received considerable study is the {\bf inclusion matrix} of $r$-subsets vs. $s$-subsets.  The study of algebraic invariants of the inclusion matrix goes back to at least the $1960$s \cite{gottlieb}.   In \cite{wilson_classic}, Wilson found an elegant diagonal form for this matrix.  His motivation was to apply this diagonal form to questions concerning the existence of $t$-designs. For coding theorists and others interested in the rank of integer matrices over fields of characteristic $p$ (the $p$-rank, for short), we remark that a diagonal form gives the $p$-rank of the matrix for all primes~$p$.

In this paper we examine a generalization of the inclusion matrix, where incidence can be defined as intersection in any fixed size.

\begin{firstdef}
For nonnegative integers $0 \leq \ell \leq \kr \leq \kc \leq n$, let $\A = \A_{n,\kr,\kc,\ell}$ be the matrix whose rows (columns) correspond to the subsets of $\{1, 2, \ldots, n\}$ having size $\kr$ ($\kc$, respectively), with entries $\A(A, B) = 1$ whenever $|A \cap B| = \ell$ and $\A(A, B) = 0$ otherwise.  
We refer to these as {\bf subset intersection matrices}.
\end{firstdef}

If $\kr = k = \kc$ then $\A$ is the adjacency matrix for the {\bf subset intersection graph} $\G(n,k,\ell)$.  
For example, we obtain the {\bf Kneser graphs}  when $\ell = 0$ and the {\bf Johnson graphs}  when $\ell = k - 1$.
It is straightforward to see that every vertex in $\G(n,k,\ell)$ has the same degree $d := { {n-k} \choose {k-\ell} }{k \choose \ell}$.  

Our main result in this paper is a simple methodology for computing or describing a diagonal form of the {\bf generalized Laplacian} matrix
\[ \L = \A - \lambda \I \]
associated to $\G(n,k,\ell)$, where $\lambda$ is any integer and $\I$ is the ${ n \choose k } \times { n \choose k }$ identity matrix, in a unified way.  Notice that our $\L$ includes the usual graph Laplacian for $\G(n,k,\ell)$ when $\lambda = d$ as well as the adjacency matrix itself (when $\lambda = 0$).  Many earlier partial results regarding these matrices have been obtained for specific values of $\ell$ and $k$; see \cite{line, brouwer:prank, duceysinhill, ds, wilson_classic} for example.  We are able to recover many of these, under one framework, as well as provide new formulas that were not previously known.  

As a byproduct of our development, we also obtain diagonal forms for the $\A_{n,\kr,\kc,\ell}$ matrices even when $\kr$ is not necessarily equal to $\kc$.  This generalizes the main result of \cite{wilson_classic} on subset inclusion matrices (where $\ell = \kr$) to arbitrary intersection parameters $\ell$.  Even more remarkable is that our techniques apply to all integer matrices in the $\Z$-span of the $\left \{ \A_{n,\kr,\kc,\ell}\right \}_{\ell = 0}^{\kr}$.  When $\kr = \kc$, these are the integer matrices in the Bose-Mesner algebra of the {\bf Johnson association scheme}.

Our methodology proceeds in three steps.  In Section~\ref{s:p_matrix}, we conjugate $\L$ by a matrix $\P$ that has previously appeared in work of Bier.  The resulting matrix $\U - \lambda \I$ is upper triangular.  Moreover, $\U$ has a block structure that is independent of the parameters $n$, $\kr$, $\kc$, and $\ell$, in the sense that the zero entries of the blocks remain the same for all parameter choices and the non-zero entries in each block all have the same value (which does, of course, depend on the parameters; see Definition~\ref{d:cf}).  In Section~\ref{s:e_matrix}, we show how to diagonalize each of these blocks simultaneously using a new unimodular matrix construction.  Finally, in Section~\ref{s:diagonalize} we combine the ingredients resulting from these steps to present our unified formula for the diagonal form of $\L$ in Definition~\ref{d:main} and Theorem~\ref{t:main}.  In Section~\ref{s:examples}, we show how this formula specializes in some familiar and new special cases.

\section{First step:  The $\P$-matrix of Bier block-triangularizes $\L$}\label{s:p_matrix}

To simplify our exposition, we assume $n$ is fixed and may refer to a ``subset'' without explicitly identifying the universal set $\{1, 2, \ldots, n\}$ to which it belongs.  Fix the graph parameters $\kk = (\kr, \kc)$ and $\ell$ as well.  Readers interested in the case when $\kr = \kc$ may simply ignore the typographical decorations on $k$ in the formulas and results that follow.

\begin{definition}\label{d:standard}
We say that a set $\beta = \{b_1, b_2, \ldots, b_k\} \subseteq \{1, 2, \ldots, n\}$ is {\bf standard} if $b_i \geq 2i$ for all $i$, where the $b_i$ form an ordered labeling of the elements (so $b_1 < b_2 < \ldots < b_k$).
\end{definition}

As an example, here are the standard subsets of $\{1, 2, 3, 4, 5\}$:
\[
\emptyset, \{2\}, \{3\}, \{4\}, \{5\}, \{2,4\}, \{2,5\}, \{3,4\}, \{3,5\}, \{4,5\}.
\]

This definition and nomenclature come from the combinatorial representation theory of the symmetric group, in which we view a subset as the second row of a two-row tableau (with the complementary entries in the first row, using the ``English'' convention).  Then, tableau are standard when their entries increase along rows and columns.  Standard subsets have also been studied under the names ``$t$-subsets of rank $t$'' \cite{frankl,bier} and ``$\ell$-tags'' \cite{singhi}.  We will try to use greek letters for standard subsets in order to differentiate them.  Also, observe that any subset of a standard subset will be standard as well.

In this section, we frequently encounter two particular classes of subsets:  {\bf unrestricted $k$-subsets} whose size must be exactly $k$, and {\bf standard $(\leq k)$-subsets} that must satisfy Definition~\ref{d:standard} as well as have a size that is no larger than $k$.  It is not too difficult to see that these two classes actually contain the same number of subsets.  In fact, the inclusion matrix between them is unimodular.

\begin{definition}\label{d:p}
Let $\P_k$ be the matrix whose rows are indexed by unrestricted $k$-subsets and whose columns are indexed by standard $(\leq k)$-subsets with entries
\[ \P_k(A, \beta) = \begin{cases} 1 & \text{ if $\beta \subseteq A$ } \\ 0 & \text{ otherwise. } \end{cases} \]
\end{definition}

\begin{theorem}
The matrix $\P_k$ is square and unimodular.
\end{theorem}
\begin{proof}
By the hook-length formula, the number of standard $s$-subsets is $\binom{n}{s} - \binom{n}{s-1}$.  Thus $\P_k$ has $\sum_{s=0}^{k} \binom{n}{s} - \binom{n}{s-1} = \binom{n}{k}$ columns, and so is square.  That $\P_k$ is unimodular (and can be used to obtain a diagonal form for the unrestricted subset-inclusion matrices) was first discovered by Bier \cite{bier}.  See also \cite[Theorem 8]{resilience} for an alternative proof of this fact.
\end{proof}

We now describe the parameters that appear in our upper triangular block form for $\L$.  These seem related (though are not precisely equal to) Eberlein orthogonal polynomials that arise as eigenvalues in the Johnson association scheme.

\begin{definition}\label{d:cf}
    Let 
    \[ c_i = c_i(j; n, \kk, \ell) = { \kr-i \choose \ell-i}{n-\kr-j+i \choose \kc-\ell-j+i} \]
    and define
    \[ f_i = f_i(j; n, \kk, \ell) = \sum_{v=0}^i (-1)^{i+v} {i \choose v} c_v(j; n, \kk, \ell). \]
\end{definition}

\begin{lemma}\label{l:ap_entries}
Let $\A = \A_{n,\kk,\ell}$ be the subset intersection matrix and $\P = \P_{\kc}$ be the Bier
matrix as in Definition~\ref{d:p}.  Suppose that $A$ is an unrestricted
$\kr$-subset and that $\beta$ is a standard $(\leq \kc)$-subset.  Then, the $(A,
\beta)$ entry of the product $\A\P$ is given by $c_a(b; n, \kk, \ell)$ where $a = |A
\cap \beta|$ and $b = |\beta|$.
\end{lemma}
\begin{proof}
We compute the dot product of row $A$ from $\A$ with column $\beta$ from $\P$.  The entries of these vectors are all zero or one, and the dot product is a summation over unrestricted $\kc$-sized subsets $T$.  Such a subset $T$ contributes one to the dot product precisely when:
\begin{itemize}
    \item $T$ intersects $A$ in a size $\ell$ subset (by the row relation from $\A$), and
    \item $T$ contains $\beta$ (by the column relation from $\P$).
\end{itemize}
Thus, the $(A, \beta)$ entry of $\A \P$ is just the number of these subsets $T$.  Let $\alpha = A \cap \beta$ with size $a := |A \cap \beta|$.  

\begin{figure}
\centering
    \begin{tikzpicture}
    \node[rectangle, minimum width = 1cm, minimum height = 0.65cm, fill=lightgray] (h) at (-0.5,0) {};
    \node[rectangle,draw, minimum width = 2cm, minimum height = 0.65cm] (a) at (0,0) {$\beta$};     
    \node[rectangle,draw, minimum width = 1cm, minimum height = 0.65cm, fill=lightgray] (b) at (1.5,0) {};
    \node[rectangle,draw, minimum width = 2.2cm, minimum height = 0.65cm] (c) at (3.1,0) {};
    \draw[decoration={brace,amplitude=2mm,mirror,raise=1mm},decorate] (b.south west) -- (b.south east);
    \draw[decoration={brace,amplitude=2mm,mirror,raise=1mm},decorate] (c.south west) -- (c.south east);
    \draw[decoration={brace,amplitude=2mm,raise=1mm},decorate] (b.north west) -- (b.north east);
    \node[above=2mm of b] (textnode) {$\ell-a$};
    \draw[decoration={brace,amplitude=2mm,raise=4mm},decorate] (a.west) -- (h.east);
    \node[above=3mm of h] (textnode) {$a$};
    \node[below=5mm of b] (comb1) {$\binom{\kr-a}{\ell-a}$};
    \node[below=5mm of c] (comb2) {$\binom{n-b-\kr+a}{\kc-b-\ell+a}$};
    \node[rectangle, minimum width = 2cm, minimum height = 0.65cm, fill=lightgray] (g) at (7.25,0) {};
    \node[rectangle, draw, minimum width = 1.0cm, minimum height = 0.65cm] (d) at (6.75,0) {$\alpha$};
    \draw[decoration={brace,amplitude=2.0mm,raise=1mm},decorate] (d.north west) -- (d.north east);
    \node[above=4mm of d] (textnode) {$a$};
    \node[rectangle,draw, minimum width = 3.5cm, minimum height = 0.65cm] (e) at (9.0,0) {};
    \draw[decoration={brace,amplitude=2.0mm,raise=1mm},decorate] (d.north east) -- (e.north east);
    \node[above=4mm of e] (textnode) {$\kr-a$};
    \node[left=2mm of h] (comb1) {$T = $};
    \node[left=2mm of g] (comb1) {$A = $};
\end{tikzpicture}
\caption{Structure of the $T$ subsets}\label{f:t_structure}
\end{figure}

Next, consider Figure~\ref{f:t_structure} where the left side represents a schematic for our possible $T$s and the right side represents the fixed $\kr$-sized set $A$. The shaded section represents the intersection between the two subsets which must have a total length of $\ell$.

Assessing these constraints, we find that $T$ is determined by the choice of entries in two blocks.  For the first block, we must choose the $\ell-a$ additional elements to lie in the intersection $A \cap T$.  These $\ell-a$ elements must be chosen from the $\kr-a$ elements of $A$ that are not already part of the intersection $\alpha$.  Thus, we have $\binom{\kr-a}{\ell-a}$ choices for the first block.

Then, we must choose a second block of $\kc-b-(\ell-a)$ elements for $T$ that do not intersect $A$ nor $\beta$.  It suffices to select elements that are disjoint from $\beta$ and from $A \setminus \alpha$, and these two sets are themselves disjoint, so there are $n-b-(\kr-a)$ possible elements to choose from.  Thus, we have $\binom{n-b-\kr+a}{\kc-b-\ell+a}$ ways to choose the second block, which yields the desired formula.
\end{proof}

We are now in a position to describe the upper triangular block form for $\L$.  

\begin{definition}
For $0 \leq i < j$, let $\W_{i,j}$ be the matrix whose rows are indexed by standard subsets of size $i$ and whose columns are indexed by standard subsets of size $j$ with entries $\W_{i,j}(\alpha, \beta) = 1$ whenever $\alpha \subseteq \beta$, and $\W_{i,j}(\alpha, \beta) = 0$ otherwise.  This same definition extends to declare that $\W_{i,j}$ is the identity matrix for $i = j$ and is the zero matrix when $i > j$.
\end{definition}

\begin{theorem}\label{t:p_matrix}
The matrix $\left(\P_{\kr}\right)^{-1} \A \P_{\kc}$ decomposes into a block matrix, where the $(i, j)$th block, for $(0, 0) \leq (i, j) \leq (\kr, \kc)$, is equal to $f_i(j) \ \W_{i,j}$.  Here, the $f_i(j)$ are as in Definition~\ref{d:cf}.

Consequently, in the case when $\kr = k = \kc$, the matrix $\P_k^{-1} \L \P_{k}$ has a decomposition into a $(k+1) \times (k+1)$ block matrix, where the $(i, j)$th block is equal to:
\begin{itemize}
    \item $f_{i}(j) \ \W_{i,j}$ for $i < j$,
    \item $\left(f_{i}(i) - \lambda\right) \ \W_{i,i}$ for $i = j$,
    \item the appropriately-sized zero matrix, otherwise.
\end{itemize}
Since the $\W_{i,i}$ are diagonal, we have that $\P_k^{-1} \L \P_k$ is an upper triangular matrix.
\end{theorem}
\begin{proof}
Define a block triangular matrix $\U$ from the $f_i(j) \ \W_{i,j}$, for $(0, 0) \leq (i, j) \leq (\kr, \kc)$.  We 
seek to show that $\A \P_{\kc} = \P_{\kr} \U$.  Then, it follows that 
\[ (\A - \lambda \I) \P_k = \A \P_k - \lambda \P_k = \P_k \U - \lambda \P_k = \P_k (\U - \lambda \I) \]
in the case $\kr = k = \kc$ so $\P_k^{-1} \L \P_k = (\U - \lambda \I)$ as in the statement.

To this end, fix an unrestricted $\kr$-subset $A$ and a standard $(\leq \kc)$-subset $\beta$.
By Lemma~\ref{l:ap_entries}, we know the $(A, \beta)$-entry of $\A \P_{\kc}$
explicitly.  Next, consider the $(A, \beta)$-entry of $\P_{\kr} \U$.
As in the previous proof, we compute the dot product of row $A$ from $\P_{\kr}$ with column $\beta$ from $\U$.  This dot product is a summation over standard $(\leq \kr)$-subsets $\sigma$.  Such a subset $\sigma$ contributes $f_i(j)$ to the dot product precisely when:
\begin{itemize}
    \item $\sigma$ is a subset of $A$ (by the row relation from $\P_{\kr}$), and
    \item $\sigma$ is a subset of $\beta$ (by the relations from the blocks $\W_{i,j}$ encountered in a fixed column of $\U$).
\end{itemize}

It is straightforward to check that the condition of being standard from Definition~\ref{d:standard} is closed under taking subsets, so $A \cap \beta$ is standard, and we sum over the subsets $\sigma$ of $A \cap \beta$.  Breaking on the size $w$ of $\sigma$, we therefore find that the $(A, \beta)$ entry of $\P_{\kr} \U$ can be written 
\[ \sum_{w = 0}^a {a \choose w} f_w(b) \]
where $a = |A \cap \beta|$ and $b = |\beta|$.  Expanding $f_w(b)$ via Definition~\ref{d:cf}, we obtain
\[ \sum_{w = 0}^a {a \choose w} \sum_{v=0}^w (-1)^{w+v} {w \choose v} c_v(b; n, \kk, \ell) \]
and interchanging (finite) summations yields
\[ = \sum_{v = 0}^a c_v(b; n, \kk, \ell) \left(\sum_{w=v}^a (-1)^{w+v} {a \choose w}  {w \choose v}\right). \]
For each fixed value of $v$, we have
$\sum_{w=v}^a (-1)^{w+v} {a \choose w}  {w \choose v} = \frac{(-1)^v}{v!} \sum_{w = v}^a (-1)^{w} \frac{a!}{(a-w)!(w-v)!}$.
When $v = a$, this summation is $1$.  To evaluate for $0 \leq v < a$, change variables to $u = w-v$, obtaining
$\frac{(-1)^v}{v!} \sum_{u = 0}^{a-v} (-1)^{u+v} \frac{a!}{(a-v-u)!(u)!} = {a \choose v} \sum_{u = 0}^{a-v} (-1)^{u} {a-v \choose u}$,
which is simply ${a \choose v} (1-1)^{a-v} = 0$.

Thus, our $(A,\beta)$-entry of $\P_{\kr}\U$ is equal to $c_a(b; n, \kk, \ell)$, which matches the $(A, \beta)$-entry of $\A\P_{\kc}$ by Lemma~\ref{l:ap_entries}, as desired.
\end{proof}

\section{Second step:  The $\E$-matrix diagonalizes each block in the triangular form}\label{s:e_matrix}

\subsection{Some historical context}

In order to proceed, we simultaneously diagonalize all of the $\W_{i,j}$ blocks that appeared after conjugating $\A$ by $\P$ in Section~\ref{s:p_matrix}. 

In \cite{wilson_classic}, Wilson found a diagonal form for the inclusion matrix, $\W^{\mathsf{classic}}_{i,j}$ say, of {\em unrestricted} $i$-subsets versus {\em unrestricted} $j$-subsets using a rather subtle induction argument.  Remarkably, the same diagonal form, with only adjustments to the multiplicities of the entries, serves as a diagonal form for our inclusion matrix $\W_{i,j}$ of {\em standard} $i$-subsets versus {\em standard} $j$-subsets.

Bier \cite{bier}, building upon work of Frankl \cite{frankl}, observed that one could bypass the induction by constructing suitable unimodular change-of-basis matrices directly.  In fact, the principal player has already been introduced in our narrative.  The $\P_i$ matrix from Definition~\ref{d:p} is capable of changing basis from unrestricted $i$-subsets to standard $(\leq i)$-subsets in essentially two different ways:  in this paper, we conjugate $\A$ by $\P_i$ to obtain a triangular form; but Bier introduced and applied $\P_i^{\mathsf{transpose}}$ on the left and $\left(\P_j^{\mathsf{transpose}}\right)^{-1}$ on the right of $\W^{\mathsf{classic}}_{i,j}$ to obtain a diagonal form (implicitly indexed now by {\em standard} subsets).  Subsequently, this matrix found application in the calculation of the Smith group of the $n$-cube graph \cite{hypercube} and in the proof of the resilience of rank of the inclusion matrices \cite{resilience}.

Thus, it is tempting to imagine that we could diagonalize our $\W_{i,j}$ matrices by a similarly explicit unimodular change-of-basis, and give a combinatorial description of the $E_s$ matrices that we construct in this section inductively.  It emerges that one should define some {\em super}-standard subset of the standard objects, and build the inclusion matrix between the standard $i$-subsets and the {\em super}-standard $(\leq i)$-subsets.  To carry this out rigorously, one would need to (a) define precisely which subsets satisfy the {\em super}-standard condition, and (b) prove that the associated inclusion matrix is actually unimodular.  We did not initially pursue this as a method of proof, but see the Appendix for results in this direction.

\subsection{The index of an integer matrix}

Given an integer matrix $M$, let $\rowz(M)$ denote the set of $\Z$-linear combinations of the rows of $M$, let $\rowq(M)$ denote the set of $\Q$-linear combinations of the rows of $M$, and let $Z(M)$ denote the set of integer vectors in $\rowq(M)$.  We have
\[
\rowz(M) \subseteq Z(M) \subseteq \rowq(M).
\]
The {\bf index} of $M$ is the index of $\rowz(M)$ as a subgroup of $Z(M)$, denoted $[Z(M) : \rowz(M)]$.  

\begin{example}
    Let $\vec{x_{1}}, \vec{x_{2}}, \vec{x_{3}}$ be a basis for $\Z^{3}$ and let $\vec{y_{1}}, \vec{y_{2}},\vec{y_{3}},\vec{y_{4}}$ be a basis for $\Z^{4}$.  
    
    Consider the homomorphism of free $\Z$-modules $\Z^{3} \to \Z^{4}$ defined by
    \begin{align*}
    \vec{x_{1}} &\mapsto 2\vec{y_{1}}\\
    \vec{x_{2}} &\mapsto 3\vec{y_{2}}\\
    \vec{x_{3}} &\mapsto 0,
    \end{align*}
    and which, with respect to these bases, has matrix
    \[M =
    \begin{bmatrix}
        2 & 0 & 0 & 0\\
        0 & 3 & 0 & 0\\
        0 & 0 & 0 & 0
    \end{bmatrix}.
    \]
    (Recall we are writing our vectors as row vectors and applying our matrices on the right.)

    Then 
    \begin{align*}
        \rowz(M) &= \{[2a,3b,0,0] \, \vert \, a,b \in \Z\} \cong 2\Z \vec{y_{1}} \oplus 3\Z \vec{y_{2}}\\
        Z(M) &= \{[a,b,0,0] \, \vert \, a,b \in \Z\} \cong \Z \vec{y_{1}} \oplus \Z \vec{y_{2}}\\
        \rowq(M) &= \{[u,v,0,0] \, \vert \, u,v \in \Q\}\\
        Z(M) / \rowz(M) &\cong \Z/2\Z \oplus \Z/3\Z,
    \end{align*}
    and the index of $M$ is $6$ (the product of the nonzero entries of the diagonal form).
   
    This quotient $Z(M) / \rowz(M)$ is the torsion subgroup of the full Smith group 
    \[
    S(M) \cong \Z/2\Z \oplus \Z/3\Z \oplus \Z^{2}.
    \]   
\end{example}

The above example shows that the index of a matrix $M$ is a useful metric when trying to unravel the Smith group, as it is always the product of positive entries in any nonnegative diagonal form for $M$.  Matrices of {\em index $\mathit 1$} play a special role.  These matrices have the property that any integer vector in the rational span of the rows $\rowq(M)$ is already in the integer span of the rows $\rowz(M)$.  A unimodular matrix is necessarily of index $1$.  Conversely, when $M$ has index $1$ then the entries in any nonnegative diagonal form of $M$ consist of $1$s and $0$s, so an index $1$ matrix may fail to be unimodular only if it is not square or does not have full rank.  We will later use the well-known fact that a full-rank matrix of index $1$ can always be enlarged to a square unimodular matrix \cite[Proposition 2]{wilson_classic}.

\subsection{Characterizing the index of $\W_{i,j}$}

We begin our analysis of $\W_{i,j}$ with a simple but fundamental result.

\begin{lemma}\label{l:fund}  For $0 \leq s \leq i \leq j$, we have
\begin{equation*}
\W_{s,i}\W_{i,j} = \binom{j-s}{i-s}\W_{s,j}.
\end{equation*}
Consequently, $\rowq(\W_{s,j}) \subseteq \rowq(\W_{i,j})$.
\end{lemma}
\begin{proof}
This is true since the $(X,Y)$-entry of the matrix product counts the number of standard $i$-subsets contained in the standard $j$-subset $Y$, and containing the standard $s$-subset $X$.  The standard subsets form an order ideal in the lattice of unrestricted subsets, partially ordered by inclusion; that is, any subset of a standard subset is itself standard.  Therefore we are just counting the number of unrestricted $i$-subsets with this condition.  This number is $\binom{j-s}{i-s}$ if $X \subseteq Y$, and is $0$ otherwise.
\end{proof}

We would like to characterize the integers that can appear in a diagonal form for $\W_{i,j}$ using index computations.  The most obvious recursion for $\W$ comes from reordering rows and columns so that we obtain a $2 \times 2$ block form based on whether a given subset contains $n$ or not.  However, this recursion runs into ``initial conditions'' that are no simpler than the general case.  To get around this difficulty, we employ a stacked matrix argument.  

\begin{definition}\label{d:stack}
For $0 \leq i \leq j$, let
\[ M_{i,j} = \bigcup_{s=0}^{i}\W_{s,j} \]
where the union denotes the matrix whose rows are the (multiset) union of the rows of the given $\W$ matrices.  We generally refer to this type of construction by saying that $M$ is a {\bf stacked} matrix.
\end{definition}

Since $\W_{i,j}$ can be found at the bottom of the matrix $M_{i,j}$, we clearly have $\rowq(\W_{i,j}) \subseteq \rowq(M_{i,j})$.  Lemma~\ref{l:fund} implies that $\rowq(M_{i,j}) \subseteq \rowq(\W_{i,j})$.  Therefore $\rowq(M_{i,j}) = \rowq(\W_{i,j})$ and the matrices $M_{i,j}$ and $\W_{i,j}$ have the same rank over $\Q$, which by dimensions is at most 
\[ \mu_{i}(n) := \binom{n}{i} - \binom{n}{i-1}, \]
the number of standard $i$-subsets of $\{1, 2, \ldots, n\}$.  (Observe that this formula only works when $n \geq 2i-1$ as there are zero standard $i$-subsets if $n < 2i$ and the formula happens to give the correct answer at $n = 2i-1$.)  

The proofs of the next two lemmas, modulo some minor technical modifications, come from ideas found in Wilson's original proof of the diagonal form for the unrestricted inclusion matrices \cite{wilson_classic}.

\begin{lemma}\label{l:mind1}
Let $0 \leq i \leq j \leq \frac{n-i}{2}$.  Then, $M_{i,j}$ has index $1$ and rank $\mu_i$.
\end{lemma}
\begin{proof}
We proceed by induction on $n$, so at times we will need to be slightly more explicit in our notation to make clear the set $\{1, 2, \cdots , n\}$ from which our subsets are taken.  Thus, we write $\W_{i,j} = \W_{i,j}(n)$ and $\mu_{i} = \mu_{i}(n) = \binom{n}{i} - \binom{n}{i-1}$.  

Clearly the result holds for the base case $i=j=0$ and $n=1$.  So assume now that $n>1$ and that $0 \leq i \leq j \leq \frac{n-i}{2}$.  We consider separately three cases that exhaust the possibilities for $i$ and $j$.

\textit{Case:} $i=0$.  Here we must have $j \leq \frac{n}{2}$ and so there are indeed standard $j$-subsets; i.e. the matrix $M_{0,j}$ does not have an empty set of columns. (This is precisely why we need the ``$2$'' in our hypothesis of the Lemma.)  In this case we have
\[
M_{0,j} = [1, 1, \cdots, 1]
\]
which has rank $\mu_{0} = 1$ and index $1$.

\textit{Case:} $i = j$.  In this case we have the identity matrix $\W_{i,i}$ found at the bottom of $M_{i,i}$.  By integral row operations all entries in $M_{i,i}$ above the bottom $W_{i,i}$ block can be cleared out using this identity matrix.  Thus in this case we see $M_{i,i}$ has rank $\mu_{i}$ and index $1$.

\textit{Case:}  $0<i<j$.  

For $0 \leq s \leq j$, we may recognize $\W_{s,j}$ as a $2 \times 2$ block matrix by ordering rows and columns of $\W_{s,j}$ so that the standard subsets containing the element $n$ appear first.  We have  
\[
\W_{s,j}(n) = 
\left [ \begin{array}{c|c}
\W_{s-1,j-1}(n-1) & 0 \\
\hline
\W_{s,j-1}(n-1) & \W_{s,j}(n-1)
\end{array} \right ].
\]

In the same way, letting $\sim$ denote equivalence by row and column operations, we obtain 
\begin{align*}
M_{i,j}(n) &\sim \left [ \begin{array}{c|c}
M_{i-1,j-1}(n-1) & 0 \\
\hline
\hline
M_{i,j-1}(n-1) & M_{i,j}(n-1)
\end{array} \right ] 
= \left [ \begin{array}{c|c}
M_{i-1,j-1}(n-1) & 0\\
\hline
\hline
M_{i-1,j-1}(n-1) & M_{i-1,j}(n-1)\\
\hline
\W_{i,j-1}(n-1) & \W_{i,j}(n-1) 
\end{array} \right ] \\
&\sim \left [ \begin{array}{c|c}
M_{i-1,j-1}(n-1) & 0\\
\hline
\hline
\W_{i,j-1}(n-1) & \W_{i,j}(n-1) \\
\hline
M_{i-1,j-1}(n-1) & M_{i-1,j}(n-1)
\end{array} \right ] 
\sim \left [ \begin{array}{c|c}
M_{i-1,j-1}(n-1) & 0\\
\hline
\hline
\W_{i,j-1}(n-1) & \W_{i,j}(n-1) \\
\hline
0 & M_{i-1,j}(n-1)
\end{array} \right ] \\
&= \left [ \begin{array}{c|c}
M_{i,j-1}(n-1) & 
\begin{array}{c}
0 \\
\hline
\W_{i,j}(n-1) \end{array}
\\
\hline
\hline
0 & M_{i-1,j}(n-1)
\end{array} \right ].
\end{align*}

Recall that our assumption on $M_{i,j}(n)$ is $2j+i \leq n$.  Since $0<i<j$, none of the block matrices written above are empty.  In particular, $M_{i,j-1}(n-1)$ and $M_{i-1,j}(n-1)$ satisfy the induction hypothesis (even though $M_{i,j}(n-1)$ might not!) and so we have 
\[
M_{i,j}(n) \sim \left [ \begin{array}{c|c}
\begin{array}{c|c}
I_1 & 0 \\
\hline
0 & 0
\end{array} & 
? \\
\hline
0 & \begin{array}{c|c}
I_2 & 0 \\
\hline
0 & 0
\end{array}
\end{array} \right ]
\]
where $I_1, I_2$ are identity matrices of orders $\mu_{i}(n-1)$ and $\mu_{i-1}(n-1)$, respectively.  Since $\rowq(M_{i,j}) = \rowq(\W_{i,j})$, we know that $M_{i,j}(n)$ has rank at most $\mu_{i}(n)$, whereas the integrally equivalent matrix from above yields that $M_{i,j}(n)$ has rank at least 
\[
\mu_{i}(n-1) + \mu_{i-1}(n-1) =  \binom{n-1}{i} - \binom{n-1}{i-1} + \binom{n-1}{i-1} - \binom{n-1}{i-2} = \binom{n}{i} - \binom{n}{i-1} = \mu_{i}(n).
\]
Thus, the rank of $M_{i,j}(n)$ must be exactly $\mu_{i}(n)$.  

Further integral row/column operations reduce $M_{i,j}(n)$ to having an identity principal submatrix $I$ of order $\mu_{i}(n)$, from which we can clear out the blocks to the right and below:
\[ M_{i,j}(n) \sim \left [ \begin{array}{c|c}
I & 0 \\
\hline
0 & ?
\end{array} \right ]. \] 
In order to have the correct rank, the remaining diagonal block must therefore be zero, so $M_{i,j}(n)$ has index $1$.
\end{proof}

\begin{definition}
Let $d_{i,j}(n)$ denote $\prod_{s=0}^{i} \binom{j-s}{i-s}^{\mu_{s}(n) - \mu_{s-1}(n)}$, where $\mu_{s}(n)$ is \mbox{$\binom{n}{s} - \binom{n}{s-1}$}.
\end{definition}

We will eventually show that the factors in this formula serve as entries in a diagonal form for $\W_{i,j}$.  At present, we are in position to prove the following weaker result.

\begin{lemma}\label{thm:index}
Let $0 \leq i \leq j \leq \frac{n-i}{2}$.  The index of $\W_{i,j}$ divides $d_{i,j}(n)$.
\end{lemma}
\begin{proof}
Let $0 \leq i \leq j \leq \frac{n-i}{2}$.  We will first show that there exist integer matrices $F_{0,j}, F_{1,j}, \cdots , F_{i,j}$ such that:
\begin{itemize}
    \item[(A)] the rows of $F_{s,j}$ lie in $\rowz(\W_{s,j})$, and
    \item[(B)] the rows of the stacked matrix $\bigcup_{s=0}^t F_{s,j}$ are a $\Z$-basis for $\rowz(M_{t,j})$, whenever $t \leq i$.
\end{itemize}
To this end, let $F_{0,j} = \W_{0,j}$.  If $i=0$ we are done.  Otherwise, inductively, assume that $F_{0,j}, F_{1,j}, \cdots , F_{t, j}$ have been defined for $t < i$.  Then $F_{0,j} \cup \cdots \cup F_{t,j}$ has index $1$, being a $\Z$-basis for $\rowz(M_{t,j})$, so by \cite[Proposition 2]{wilson_classic} we can adjoin $\mu_{t+1}-\mu_{t}$ row vectors from $\rowz(\W_{t+1,j})$ to get a $\Z$-basis for $\rowz(M_{t+1,j}) = \rowz(M_{t,j}) + \rowz(\W_{t+1,j})$.  These additional rows are the rows of $F_{t+1,j}$.  This proves the existence of the sequence $\{F_{s,j}\}_{s=0}^{i}$ with Properties (A) and (B) above.

Next, we claim that the rows of the ``inflated'' stacked matrix $\bigcup_{s=0}^{i}\binom{j-s}{i-s} F_{s,j}$ lie in $\rowz(\W_{i,j})$.  To see this, let $0 \leq s \leq i$.  By Property (A) the rows of $F_{s,j}$ lie in $\rowz(\W_{s,j})$, so there exists an integer matrix $C$ so that $C \W_{s,j} = F_{s,j}$.  Multiplying on both sides by $\binom{j-s}{i-s}$ and using Lemma~\ref{l:fund}, we obtain
\[ \binom{j-s}{i-s} F_{s,j} = \binom{j-s}{i-s} C \W_{s,j} = C (\W_{s,i} \W_{i,j}) = \left(C \W_{s,i}\right) \W_{i,j}. \]

Finally, we have $d_{i,j} = [\rowz(M_{i,j}): \bigcup_{s=0}^{i}\binom{j-s}{i-s} F_{s,j}]$ by Property (B) for the $F$ matrices above.
On the other hand, we have as a consequence of Lemma~\ref{l:mind1} that $Z(M_{i,j}) = \rowz(M_{i,j}) = Z(\W_{i,j})$.  Thus, 
\[ d_{i,j} = [Z(\W_{i,j}) : \bigcup_{s=0}^{i}\binom{j-s}{i-s} F_{s,j}] = [Z(\W_{i,j}) :   \rowz(\W_{i,j}) ]   \cdot  [\rowz(\W_{i,j}) :  \bigcup_{s=0}^{i}\binom{j-s}{i-s} F_{s,j}]. \]
Thus, the index of $\W_{i,j}$ divides $d_{i,j} = \prod_{s=0}^{i} \binom{j-s}{i-s}^{\mu_{s} - \mu_{s-1}}$.
\end{proof}

\subsection{Constructing the diagonal form for $\W_{i,j}$}

In this section we will need the hypothesis $0 \leq i \leq j \leq \frac{n-i}{2}$ of the propositions we developed so far to be satisfied for all $i \leq j$, and in particular when $i = j-1$.  Thus, we need $j \leq \frac{n+1}{3}$, and this will be the hypothesis for the following results.

\begin{definition}\label{d:diag_for_w}
For $0 \leq i \leq j \leq \frac{n+1}{3}$, let $D_{i,j}$ denote the $\mu_{i} \times \mu_{j}$ diagonal matrix with diagonal entries given by 
\[ \left\{ {j-s \choose i-s}^{\mu_s - \mu_{s-1}} : 0 \leq s \leq i \right\} \]
where exponents denote multiplicity of the diagonal entry.  Similarly, we use $D_{i,j}^{\prime}$ to denote the $\mu_{i} \times \mu_{i}$ diagonal matrix containing the same entries; that is, $D_{i,j}^{\prime}$ is just $D_{i,j}$ with the zero columns removed.
\end{definition}

The proof of the following result contains the main construction of this section.

\begin{theorem}\label{t:bier-std}
For each $0 \leq s \leq \frac{n+1}{3}$, there exists a unimodular $\mu_s \times \mu_s$ matrix $E_{s}$ such that whenever $s \leq r \leq \frac{n+1}{3}$, we have
\begin{equation}\label{eqn:recursion}
    E_{s}\W_{s,r} = D_{s,r}^{\prime} E_{s,r}^{\prime} \tag{*}
\end{equation}
for some $\mu_{s} \times \mu_{r}$ integer matrix $E_{s,r}^{\prime}$.
\end{theorem}
\begin{proof}
We suppose $n$ is fixed and proceed by induction on $s$.  First, define $E_{0} = [1]$.  Then for all $r \geq 0$ we have
\begin{align*}
E_{0}\W_{0,r} &= [1] \begin{bmatrix} 1 & 1 & \cdots & 1 \end{bmatrix}\\
&= D_{0,r}^{\prime} E_{0,r}^{\prime}.
\end{align*}

Next, assume that $E_{s}$ has been constructed and that it satisfies Property~(\ref{eqn:recursion}) for all valid $r$.  Then, define $E_{s,s+1}^\prime$ via the equation
\begin{align} \label{eqn:index1}
    E_{s}\W_{s,s+1} &= D_{s,s+1}^{\prime} E_{s,s+1}^{\prime}.  \tag{**}
\end{align}

By the induction hypothesis, $E_{s,s+1}^{\prime}$ is an integer matrix, and we claim that it has index $1$.  Assume for a moment that this is true so we can explain how to complete the induction.  By \cite[Proposition 2]{wilson_classic}, we may adjoin additional rows to $E_{s,s+1}^{\prime}$ to form a $\mu_{s+1} \times \mu_{s+1}$ unimodular matrix.  Do so, and define this matrix to be $E_{s+1}$.  Let $B$ denote the $(\mu_{s+1}-\mu_{s}) \times \mu_{s+1}$ matrix containing the rows that we adjoined to $E_{s,s+1}^{\prime}$ to build $E_{s+1}$.  

Now, to check that $E_{s+1}$ continues to satisfy Property~(\ref{eqn:recursion}), fix $r \geq s+1$.  We have
\begin{align*}
    E_{s+1} \W_{s+1,r} &= \left [\begin{array}{c} E_{s, s+1}^{\prime}\\ \hline
    B
    \end{array} \right ] \cdot \W_{s+1,r}
    = 
    \left [\begin{array}{r} E_{s, s+1}^{\prime} \cdot \W_{s+1,r} \\ \hline
    B \cdot \W_{s+1,r} 
    \end{array} \right ]\\
    &= 
    \left [\begin{array}{r} \left ( D_{s,s+1}^{\prime} \right )^{-1} E_{s} \W_{s,s+1} \cdot \W_{s+1,r} \\ \hline
    B \cdot \W_{s+1,r} 
    \end{array} \right ],\ \mbox{ by definition of $E_{s, s+1}^{\prime}$}\\
    &= 
    \left [\begin{array}{r} \left ( D_{s,s+1}^{\prime} \right )^{-1} E_{s} \cdot (r-s) \W_{s,r} \\ \hline
    B \cdot \W_{s+1,r} 
    \end{array} \right ],\ \mbox{ by Lemma \ref{l:fund}}\\
    &= 
    \left [\begin{array}{r} \left ( D_{s,s+1}^{\prime} \right )^{-1}  (r-s) D_{s,r}^{\prime} \cdot E_{s,r}^{\prime} \\ \hline
    B \cdot \W_{s+1,r} 
    \end{array} \right ],\ \mbox{ by Property~(\ref{eqn:recursion}).}
\end{align*}
Examining carefully, we see that $\left (D_{s,s+1}^{\prime} \right )^{-1}  (r-s) D_{s,r}^{\prime}$ is a $\mu_{s} \times \mu_{s}$ diagonal matrix with entries given by
\begin{align*}
\left[ \frac{r-s}{s+1-t} \cdot \binom{r - t}{s - t} \right]^{\mu_{t}-\mu_{t-1}} 
= \binom{r - t}{s +1 - t}^{\mu_{t}-\mu_{t-1}},
\end{align*}
for $0 \leq t \leq s$.  If we enlarge this to a diagonal matrix of size $\mu_{s+1} \times \mu_{s+1}$ by adding $\mu_{s+1} - \mu_{s}$ diagonal entries equal to $1$, then we get precisely the matrix $D_{s+1,r}^{\prime}$.  Therefore
\[
E_{s+1} \W_{s+1,r} = D_{s+1,r}^{\prime} \cdot \left [\begin{array}{c} E_{s,r}^{\prime} \\ \hline
    B \cdot \W_{s+1,r} 
    \end{array} \right ],
\]
as desired.

Thus, it remains only to check that $E_{s,s+1}^{\prime}$ has index $1$, and this is where our earlier results on the index of $\W_{i,j}$ come into play.  Recall that $d_{s,t}$ denotes the product of all of the diagonal entries in $D_{s,t}^{\prime}$, and observe that:
\begin{itemize}
\item By Lemma~\ref{thm:index}, we know $\W_{s,s+1}$ has index dividing $d_{s,s+1}$, so we obtain 
\[ d_{s,s+1} \geq [Z(\W_{s,s+1}) : \rowz(\W_{s,s+1})]. \]

\item Since $E_{s}$ is unimodular, any integer linear combination $\vec{x}\W_{s,s+1}$ of the rows of $\W_{s,s+1}$ can be written
\[
\vec{x}\W_{s,s+1} = \vec{x}E_{s}^{-1} E_{s}\W_{s,s+1} = \vec{y}E_{s}\W_{s,s+1},
\]
and so we have $\rowz(\W_{s,s+1}) = \rowz(E_{s}\W_{s,s+1})$. 

\item Since $\W_{s,s+1}$ has full row rank, Equation~(\ref{eqn:index1}) shows that the rows of $E_{s,s+1}^{\prime}$ are linearly independent and therefore are a $\Z$-basis for their $\Z$-span $\rowz(E_{s,s+1}^{\prime})$.  Therefore 
\[ [\rowz(E_{s,s+1}^{\prime}) : \rowz(E_{s}\W_{s,s+1})] = [\rowz(E_{s,s+1}^{\prime}) : \rowz(D_{s,s+1}^{\prime} E_{s,s+1}^{\prime})] = d_{s,s+1}. \]
\end{itemize}
Putting these facts together, we have
\begin{align*}
    d_{s,s+1} &\geq [Z(\W_{s,s+1}) : \rowz(\W_{s,s+1})]\\
    &= [Z(\W_{s,s+1}) : \rowz(E_{s}\W_{s,s+1})]\\
    &= [Z(\W_{s,s+1}) : \rowz(E_{s,s+1}^{\prime})] \cdot [\rowz(E_{s,s+1}^{\prime}) : \rowz(E_{s}\W_{s,s+1})]\\
    &= [Z(\W_{s,s+1}) : \rowz(E_{s,s+1}^{\prime})] \cdot d_{s,s+1}
\end{align*}
and so $[Z(\W_{s,s+1}) : \rowz(E_{s,s+1}^{\prime})]=1$.

Finally, we claim that $Z(\W_{s,s+1}) = Z(E_{s,s+1}^{\prime})$.  It follows from Equation~(\ref{eqn:index1}) that the rows of $E_{s,s+1}^{\prime}$ are in $\rowq(\W_{s,s+1})$.  Therefore $\rowq(E_{s,s+1}^{\prime}) \subseteq \rowq(\W_{s,s+1})$, and so $Z(E_{s,s+1}^{\prime}) \subseteq Z(\W_{s,s+1})$.
Thus, we have
\begin{align*}
1 &= [Z(\W_{s,s+1}) : \rowz(E_{s,s+1}^{\prime})]\\
&= [Z(\W_{s,s+1}) : Z(E_{s,s+1}^{\prime})] \cdot [Z(E_{s,s+1}^{\prime}) : \rowz(E_{s,s+1}^{\prime})].
\end{align*}
Therefore, $[Z(E_{s,s+1}^{\prime}) : \rowz(E_{s,s+1}^{\prime})]=1$ and so $E_{s,s+1}^{\prime}$ has index $1$.
\end{proof}

\begin{remark}
To summarize the construction, we begin with the matrix $E_{0}=[1]$, and then each successive unimodular matrix $E_{j+1}$ comes from applying $\W_{j,j+1}$ to $E_{j}$, dividing each row by an integer to make the entries of that row relatively prime, and then extending this matrix to a unimodular one.  In all cases we have computed, it is possible to achieve this unimodular extension to $E_{j+1}$ by adjoining certain rows of the identity matrix.  This seems to suggest that these matrices could be obtained by selecting certain ``super-standard'' subsets.
\end{remark}

Finally, we complete the program we outlined at the beginning of this section by observing that the sequence $\{E_{s}\}$ turns out to have a very nice index-uniformity property.

\begin{theorem} \label{t:standard_bier}
Let $0 \leq i \leq j \leq \frac{n+1}{3}$, and let $\{E_{s}\}$ be the sequence of square unimodular matrices defined in Theorem~\ref{t:bier-std}.  Then,
\begin{equation*}
        E_{i}\W_{i,j} = D_{i,j}E_{j}.
\end{equation*}
\end{theorem}
\begin{proof}
We proceed by induction on the sum $i+j$.  First, notice that the result is clearly true when $i=j=0$, and in fact when $i=j \leq \frac{n+1}{3}$ (since $\W_{i,i} = D_{i,i}$ is the identity matrix).
    
So assume that $i < j \leq \frac{n+1}{3}$.  We have
\begin{align*}
        E_{i} \W_{i,j} &= \frac{1}{j-i} E_{i} \W_{i,j-1} \W_{j-1,j}, \mbox{ by Lemma \ref{l:fund}}\\
        &= \frac{1}{j-i} D_{i,j-1} E_{j-1} \W_{j-1,j}, \mbox{ by our induction hypothesis}\\
        &= \frac{1}{j-i} D_{i,j-1} D_{j-1,j}^{\prime} E_{j-1,j}^{\prime}, \mbox{ by the defining property in Theorem~\ref{t:bier-std}. }
\end{align*}
Following the construction given in the proof of Theorem~\ref{t:bier-std}, we have that $D_{j-1,j}^{\prime} E_{j-1,j}^{\prime}$ is $D_{j-1,j} E_{j}$.  Here, $D^{\prime}$ is replaced by $D$ to compensate for the change in the dimensions that results when we add rows to $E_{j-1,j}^{\prime}$ to obtain $E_j$.  Thus, we obtain
\[ E_i \W_{i,j} = \frac{1}{j-i} D_{i,j-1} D_{j-1,j} E_{j} = D_{i,j} E_{j}, \]
since $\frac{1}{j-i}\binom{j-1-s}{i - s} \binom{j - s}{j-1-s} = \binom{j-s}{i-s}$, as desired.
\end{proof}

\begin{definition}
For $0 \leq k \leq \frac{n+1}{3}$, define an $\binom{n}{k} \times \binom{n}{k}$ block diagonal matrix $\E=\E_{k}$, with diagonal blocks given by the matrices $E_{j}^{-1}$ where $0 \leq j \leq k$:
    \[
    \E = \begin{bmatrix}
        E_{1}^{-1} & 0 & 0 & \cdots &0 \\
        0 & E_{2}^{-1} &  0 & \cdots &0\\
        0 & 0 & \ddots & & \vdots \\
        \vdots & \vdots & & &\\
        0 & 0 & \cdots & & E_{k^{-1}}
    \end{bmatrix}.
    \]
Note that by construction, the matrix $\E$ is unimodular.
\end{definition}

By Theorem \ref{t:p_matrix}, the matrix $\A = \A_{n, \kr, \kc, \ell}$ (and the generalized Laplacian $\L$ when $\kr=\kc$) are integrally equivalent to a block matrix comprised of integer multiples of the standard-subset inclusion matrices $\W_{i,j}$.

\begin{corollary}\label{c:w_diagonal}
When $n$ is sufficiently large (e.g. $n \geq 3\kc - 1$), the $\W_{i,j}$ blocks in this configuration can simultaneously be brought into diagonal forms $D_{i,j}$ with entries 
\[ \left\{ {j-s \choose i-s}^{\mu_s - \mu_{s-1}} : 0 \leq s \leq i \right\} \]
where the exponent indicates the multiplicity of the diagonal entry and $\mu_s = {n \choose s} - {n \choose {s-1}}$ (with $\binom{a}{b} = 0$ whenever $b<0$ or $a<b$), by applying $\E_{\kc}$ on the right and $\E_{\kr}^{-1}$ on the left.  In the case when $\kr=\kc$, this is simply another conjugation applied to the generalized Laplacian.
\end{corollary}

In the next section we describe how to extract from this situation a diagonal form of the entire block matrix.

\section{Third step:  The matrix of diagonal blocks can itself be diagonalized using matrices of small dimension}\label{s:diagonalize}

We are now in position to present our framework for finding a diagonal form for $\L = \A_{n,\kr,\kc,\ell} - \lambda \I$.

\begin{definition}\label{d:main}
For each $s = 0, 1, \ldots, \kr$, consider the $(\kr-s+1) \times (\kc-s+1)$ matrix $\M_s$ with entries
\[ \M_s(i,j) = -\lambda \ \delta_{i,j} + {j-s \choose i-s} \sum_{v=0}^i (-1)^{i+v} {i \choose v} { \kr-v \choose \ell-v}{n-\kr-j+v \choose \kc-\ell-j+v}. \]
Here, the entries are (unconventionally, but harmlessly) indexed with $(s, s) \leq (i, j) \leq (\kr, \kc)$, and $\delta_{i,j}$ is the Kronecker delta function (which is $0$, unless $i = j$ when it is $1$).
\end{definition}

We have expanded this formula for ease of reference, independent from our development in this paper, but it may also be written more compactly as
\[ \M_s(i,j) = -\lambda \ \delta_{i,j} + {j-s \choose i-s} f_i(j; n, \kk, \ell), \]
and the reader may refer to Section~\ref{s:examples} for some simple formulas for $f_i(j; n, \kk, \ell)$ as well as some examples of $\M_s$ in various cases.

\begin{remark}\label{r:evs}
When $\kr = k = \kc$ and $\lambda$ is equal to the degree of $\G$, we find diagonal entries
\[ \M_s(i,i) = f_i(i; n, k, \ell) - { {n-k} \choose {k-\ell} }{k \choose \ell} \hspace{0.2in} (\text{for } s \leq i \leq k). \]
These are the eigenvalues of the ``classical'' Laplacian $\L$, being the diagonal entries of the triangular matrix $\P^{-1} \L \P$ by Theorem~\ref{t:p_matrix}.  Although in this presentation the $i$th eigenvalue, that we denote $e_i$, is interlaced among the $\M_s$ matrices for which $s \leq i$, the multiplicity of $e_i$ is 
$\mu_0 + ( \mu_1 - \mu_0 ) + (\mu_2 - \mu_1) + \cdots + (\mu_i - \mu_{i-1}) = \mu_i$,
as expected.
\end{remark}

\begin{figure}[h]
\centering
    \scalemath{0.7}{
    \begin{array}{c|cccc|ccccccc|cccccccccc}
    \textcolor{blue}{e_0} & \textcolor{blue}{f_{0,1}} & 0 & \cdots &  0 & \textcolor{blue}{f_{0,2}} & 0  & \cdots & \cdots &  \cdots&  \cdots& 0 & \textcolor{blue}{f_{0,3}} & 0 & \cdots & \cdots &  \cdots&  \cdots& \cdots & \cdots &  \cdots &0\\
    \hline
    \textcolor{blue}{0} & \textcolor{blue}{e_1} &0&\cdots&0&\textcolor{blue}{2f_{1,2}}& 0 & \cdots & \cdots &  \cdots&  \cdots& 0 &\textcolor{blue}{3f_{1,3}}& 0 & \cdots & \cdots &  \cdots& \cdots & \cdots &  \cdots &  \cdots& 0\\

    \vdots& 0& \textcolor{orange}{e_1} & & \vdots& 0&  \textcolor{orange}{f_{1,2}} & 0 & \cdots &  \cdots&  \cdots& 0 & 0&  \textcolor{orange}{f_{1,3}} & 0 & \cdots &  \cdots&  \cdots & \cdots &  \cdots&  \cdots& 0\\
    
    \vdots & \vdots & & \textcolor{orange}{\ddots} & 0 &\vdots & &\textcolor{orange}{\ddots}&&& & \vdots & \vdots& & \textcolor{orange}{\ddots} &&&&&&& \vdots\\
    
    0 & 0 &\cdots & 0 & \textcolor{orange}{e_1} & 0  & \cdots & 0 & \textcolor{orange}{f_{1,2}} & 0 & \cdots & 0 & 0  & \cdots & 0 & \textcolor{orange}{f_{1,3}} & 0 & \cdots & \cdots & \cdots & \cdots & 0\\
    
    \hline
    \textcolor{blue}{0} & \textcolor{blue}{0} & \cdots & \cdots & 0 & \textcolor{blue}{e_2} & 0  & \cdots & \cdots & \cdots & \cdots & 0 &\textcolor{blue}{3f_{2,3}}& 0 & \cdots & \cdots &  \cdots& \cdots & \cdots &  \cdots &  \cdots& 0\\

    \vdots & \vdots &&& \vdots & 0 & \textcolor{orange}{e_2} &&&&& \vdots & 0 & \textcolor{orange}{2f_{2,3}} & 0 & \cdots &  \cdots&  \cdots & \cdots &  \cdots&  \cdots& \vdots\\
    
    \vdots & \vdots &&& \vdots & \vdots && \textcolor{orange}{\ddots} &&&& \vdots & \vdots& & \textcolor{orange}{\ddots} &&&&&&& \vdots\\
    
    \vdots & \vdots &&& \vdots & \vdots &&& \textcolor{orange}{e_2} &&& \vdots & 0  & \cdots & 0 & \textcolor{orange}{2f_{2,3}} & 0 & \cdots & \cdots & \cdots & \cdots & 0\\
    
    \vdots & \vdots &&& \vdots & \vdots &&&& \textcolor{gray}{e_2} && \vdots & 0  & \cdots & \cdots & 0 & \textcolor{gray}{f_{2,3}} & 0 & \cdots & \cdots & \cdots & 0\\
    
    \vdots & \vdots &&& \vdots & \vdots &&&&& \ddots & 0 & \vdots &&&&& \ddots &&&& \vdots\\
    
    0 & 0 & \cdots &\cdots & 0 & 0  & \cdots & \cdots & \cdots & \cdots & 0 & \textcolor{gray}{e_2} & 0 &\cdots&\cdots&\cdots&\cdots&0& \textcolor{gray}{f_{2,3}} &0&\cdots& 0\\

    \hline
    \textcolor{blue}{0} & \textcolor{blue}{0} & \cdots &\cdots & 0 & \textcolor{blue}{0} & \cdots &\cdots &\cdots &\cdots &\cdots & 0 & \textcolor{blue}{e_3} & 0 &\cdots &\cdots &\cdots &\cdots & \cdots & \cdots & \cdots & 0\\

    \vdots & \vdots &&& \vdots & \vdots &&&&&& \vdots & 0 & \textcolor{orange}{e_3}&&&&&&&& \vdots\\

    \vdots & \vdots &&& \vdots & \vdots &&&&&& \vdots & \vdots && \textcolor{orange}{\ddots}&&&&&&& \vdots\\
    \vdots & \vdots &&& \vdots & \vdots &&&&&& \vdots & \vdots &&& \textcolor{orange}{e_3}&&&&&& \vdots\\
    \vdots & \vdots &&& \vdots & \vdots &&&&&& \vdots & \vdots &&&& \textcolor{gray}{e_3} &&&&& \vdots\\
    \vdots & \vdots &&& \vdots & \vdots &&&&&& \vdots & \vdots &&&&& \ddots&&&& \vdots\\
    \vdots & \vdots &&& \vdots & \vdots &&&&&& \vdots & \vdots &&&&&& \textcolor{gray}{e_3} &&& \vdots\\
    \vdots & \vdots &&& \vdots & \vdots &&&&&& \vdots & \vdots &&&&&&& \textcolor{black}{e_3} && \vdots\\
    \vdots & \vdots &&& \vdots & \vdots &&&&&& \vdots & \vdots &&&&&&&& \ddots& 0\\
    0 & 0 &\cdots&\cdots& 0 & 0 &\cdots&\cdots&\cdots&\cdots&\cdots& 0 & 0 &\cdots&\cdots&\cdots&\cdots&\cdots&\cdots&\cdots& 0& \textcolor{black}{e_3}\\
\end{array}  }
\caption{Schematic for $\E^{-1}\P^{-1}\L\P\E$ when $k=3$ with the $\M_s$ appearing as embedded sub-matrices}\label{f:main}
\end{figure}

\begin{example}
In Figure~\ref{f:main}, we show the $k=3$ case of the form we obtain for $\E^{-1} \P^{-1} \L \P \E$ after we have diagonalized each $\W_{i,j}$ block.  Here, we have $\M_0$ (a $4 \times 4$ matrix, appearing once) in blue; $\M_1$ (a $3 \times 3$ matrix, repeated $\mu_1 - 1$ times) in orange; $\M_2$ (a $2 \times 2$ matrix, repeated $\mu_2 - \mu_1$ times) in gray; and $\M_3$ (a $1 \times 1$ matrix, repeated $\mu_3 - \mu_2$ times) in black.  The $e_i$ appearing on the diagonal are the eigenvalues of $\L$.
\end{example}

As each $\M_s$ is an (upper-triangular) integer matrix, it has a diagonal form that can be achieved by multiplying on the left and right by, possibly different, unimodular matrices.  For example, we may compute its Smith normal form.  Let $\Delta(\M_s)$ be the multiset of entries from such a diagonal form; for example, we could compute the invariant factors.

\begin{theorem}\label{t:main}
Suppose the parameters $n$, $\kk = (k,k)$, and $\ell$ are fixed, with $n \geq 3k-1$.  Given a subset intersection graph $\G(n, k, \ell)$ and an integer $\lambda$, we find that the generalized Laplacian $\L = \A - \lambda \I$ has a diagonal form with entries
\[ \bigcup_{s = 0}^k \Delta(\M_s)^{ {n \choose s} - 2 {n \choose {s-1}} + {n \choose {s-2}}}. \]
Here, the exponent indicates that each element of $\Delta(\M_s)$ (which may itself appear multiple times) should be included in the diagonal form with the indicated multiplicity.

Using the language of abelian groups, we may equivalently write
\[
S(\L) \cong \bigoplus_{s=0}^{k} S(\M_s)^{ {n \choose s} - 2 {n \choose {s-1}} + {n \choose {s-2}}}.
\]
\end{theorem}
\begin{proof}
This follows directly from the development in the previous sections:  the expression in Definition~\ref{d:main} is the $s$th distinct diagonal entry from Corollary~\ref{c:w_diagonal} for each block, multiplied by the appropriate $f_i(j)$ from Definition~\ref{d:cf} as per Theorem~\ref{t:p_matrix}.  
The multiplicity formula is $\mu_s - \mu_{s-1}$ from Corollary~\ref{c:w_diagonal} where $\mu_s = {n \choose s} - {n \choose {s-1}}$.
See Figure~\ref{f:main} for a schematic in the case when $k=3$.

The key point is that each copy of each $\M_s$ matrix is embedded in $\E^{-1} \P^{-1} \L \P \E$ such that
\begin{itemize}
    \item The entries in each row and column of $\E^{-1} \P^{-1} \L \P \E$ lying outside of $\M_s$ are all zero, and
    \item the diagonal entries of $\M_s$ lie on the main diagonal of $\E^{-1} \P^{-1} \L \P \E$.
\end{itemize}
Therefore, we can embed the unimodular transformations (i.e. integral row/column operations) that diagonalize each copy of each $\M_s$ into a single pair of unimodular transformations that diagonalize $\E^{-1} \P^{-1} \L \P \E$.  (There is some freedom here that can be exploited depending on the graph parameters or desired formulas, but taking transformations obtained from the Smith normal form of $\M_s$ are always available as a default implementation.)
\end{proof}

\begin{remark}
Observe that finding a diagonal form for the ${n \choose k} \times {n \choose k}$ matrix $\L$ reduces to finding diagonal forms for $(k+1)$ matrices that are each no larger than $(k+1) \times (k+1)$.  This is a dramatic reduction in computational (and mathematical) complexity!
\end{remark}

\begin{remark}
In addition, the diagonal entries of the $\M_s$ are the eigenvalues of $\L$.  Thus the $\M_s$ can be viewed as a transparent and precise record of the additional information needed to proceed from diagonalizing over a field to diagonalize over the integers.
\end{remark}

\begin{corollary}\label{c:main}
Suppose the parameters $n$, $\kk = (\kr,\kc)$, and $\ell$ are fixed, with $\kr \leq \kc$ and $n \geq 3\kc-1$.
The subset intersection matrix $\A_{n, \kk, \ell}$ also has a diagonal form with entries just as in Theorem~\ref{t:main} using the generalized $\M_s$ matrices from Definition~\ref{d:main} with $\lambda = 0$.
\end{corollary}

\section{Some examples and applications}\label{s:examples}

\subsection{Matrices of the Johnson scheme}
When $\kr = k = \kc$, the matrix $\A_{n, k, \ell}$ of Corollary \ref{c:main} is the $(k-\ell)$-association matrix of the Johnson association scheme. Note that $\A_{n,k,0}$ is the adjacency matrix of the Kneser graph, $\A_{n,k, k-1}$ is the adjacency matrix of the Johnson graph, and $\A_{n,k,k}$ is the identity matrix. 

While the authors' primary motivation for this work was to study the classical graph Laplacian, there is also much interest in these association matrices and especially in properties of the algebra they generate (the Bose-Mesner algebra of the association scheme).  In \cite[Problem 3.7]{snf-inc}, Sin asks for a solution to the ``SNF problem'' for these matrices, which we have given in the theorems above and we examine more closely in this section.  

In fact, our technique extends to describe the Smith group of any integer matrix in the Bose-Mesner algebra of the Johnson scheme.  This is remarkable since knowing the Smith groups of two matrices does not normally yield any information about the Smith group of their sum.  

\begin{theorem}\label{t:j_scheme}
Let $B = \sum_{\ell = 0}^{k} b_{\ell} \A_{n,k,\ell}$, where the $b_{\ell}$ are integers.  Define the $(k-s+1) \times (k-s+1)$ matrix $\M_{s}(B)$, with $(i,j)$-entry equal to
\begin{equation}\label{eqn:scheme}
\binom{j-s}{i-s}\sum_{\ell = 0}^{k} b_{\ell} f_i(j; n, k, \ell).
\end{equation}
As in Definition \ref{d:main}, the matrix entry indices have the range $s \leq i, j \leq k$.

Then 
\[
S(B) \cong \bigoplus_{s=0}^{k} S(\M_s(B))^{ {n \choose s} - 2 {n \choose {s-1}} + {n \choose {s-2}}}.
\]
\end{theorem}
\begin{proof}
    Conjugating $B$ by the matrix $\P \E$ produces a linear combination of matrices that all share the same diagonal block-form.  Thus we get a diagonal block-form for $B$ where the block coefficients are the corresponding linear combinations in Equation~(\ref{eqn:scheme}) of the block coefficients $f_i(j; n, k, \ell)$ of the $\A_{n,k,\ell}$. 
 The same proof as of Theorem \ref{t:main} now extracts the Smith group.
\end{proof}

The strongest previous result on this topic may be \cite[Theorem 9]{wong}.  There the authors describe a diagonal form for certain matrices in the Bose-Mesner algebra of the Johnson scheme that satisfy a primitivity condition.  However, the association matrices themselves \textit{do not} usually satisfy this condition.

\begin{example}\label{ex:scheme}
Let $n=12, k=3$ and consider the association matrices $A_1 = \A_{12,3,1}$ and $A_2 = \A_{12,3,2}$.  Let $B = A_1 + 3A_2$.

We compute the $\M_{s}$ matrices for $B$:
\begin{align*}
\M_{0}(B) &= \M_{0}(A_1) + 3\M_{0}(A_2) = \begin{bmatrix}
    189 & 33 & 3 & 0\\
    0 & 57 & 22 & 3\\
    0 & 0 & 2 & 3\\
    0 & 0 & 0 & -6
\end{bmatrix} \\
\M_{1}(B) &= \M_{1}(A_1) + 3\M_{1}(A_2) = \begin{bmatrix}
    57 & 11 & 1\\
    0 & 2 & 2\\
    0 & 0 & -6
\end{bmatrix}\\
\M_{2}(B) &= \M_{2}(A_1) + 3\M_{2}(A_2) = \begin{bmatrix}
    2 & 1\\
    0 & -6
\end{bmatrix} \\
\M_{3}(B) &= \M_{3}(A_1) + 3\M_{3}(A_2) = \begin{bmatrix}
    -6
\end{bmatrix}
\end{align*}
Applying Theorem \ref{t:j_scheme} we get
\begin{align*}
S(B) &\cong S(\M_{0}(B))^{1} \oplus S(\M_{1}(B))^{10} \oplus S(\M_{2}(B))^{43} \oplus S(\M_{3}(B))^{100}\\
&\cong \left ( \Z/3\Z^{2} \oplus \Z/14364\Z \right )
\oplus \left ( \Z/2\Z \oplus \Z/342\Z \right )^{10} 
\oplus \left ( \Z/12\Z \right )^{43}
\oplus \left ( \Z/6\Z \right )^{100}.
\end{align*}
One can build directly the $220 \times 220$ matrix $B = A_{1} + 3A_{2}$ in Sage software \cite{sagemath} and see that we have indeed predicted the correct Smith group.
\end{example}

\subsection{Johnson graphs}

Although the general formula in Definition~\ref{d:main} may seem formidable, it turns out that the $f_i(j; n, k, \ell)$ summations often have a simple form for particular values of $\ell$.  For example, there are only two non-zero diagonals in the case when $\ell = k-1$.

\begin{lemma}\label{l:johnson_M}
If $\ell = k-1$, we have
\[ f_{i}(j) = \begin{cases} 
 (k-i)(n-k-i) - i   & \text{ if $i = j$ } \\
 (k-i)       & \text{ if $i = j-1$ } \\
 0       & \text{ otherwise. }
\end{cases} \]
\end{lemma}
\begin{proof}
We have

$c_i(j; n, k, k-1) = (k-i) {n-k-(j-i) \choose 1-(j-i)} 
= \begin{cases}
    (k-i) (n-k) & \text{ if $i = j$ } \\
    (k-i)  & \text{ if $i = j-1$ } \\
    0 & \text{ if $i < j-1$ } \\
    \end{cases}$

\noindent
so $f_i(j; n, k, k-1) = \sum_{v=0}^i (-1)^{i+v} {i \choose v} c_v(j; n, k, k-1)$ is $0$ unless $i \geq j-1$.
If $i = j-1$, we have
\[ f_{j-1}(j) = c_{j-1}(j) = (k-i), \]
while if $i = j$, we have
\[ f_j(j) = - j c_{j-1}(j) + c_j(j) = -j (k-(j-1)) +  (k-j) (n-k), \]
as claimed.
\end{proof}

\subsubsection{Johnson Laplacian matrix}

We find that the $\M_s$ matrices have a particularly simple structure in this case.

\begin{corollary}\label{c:johnson} 
If $\ell = k-1$ and $\lambda = d = (n-k)k$, the matrix $\M_s$ has diagonal entries $-i(n-(i-1))$ where $i$ runs from $s$ to $k$, and the entries on the super-diagonal correspond to values of the sequence $i(t-(i-1))$ as $i$ runs from $1$ to $t := k-s$.  All other entries of $\M_s$ are zero.
\end{corollary}

By Remark~\ref{r:evs}, the diagonal entries $e_0 = 0, e_1 = n, e_2 = 2(n-1), e_3 = 3(n-2), e_4 = 4(n-3), \ldots$ are the eigenvalues of $\L$.  From Corollary~\ref{c:johnson}, it is a straightforward exercise to read off the entries of a diagonal form for the classical Laplacian of the Johnson graphs.  The next two theorems demonstrate this, in detail, for $k = 2$ and $k = 3$.

\medskip
{\bf The proof of Theorem~\ref{t:firstapp} introduces a template that we will use to obtain all of the results in this section.}

\begin{theorem}\label{t:firstapp}
The classical Laplacian of the $k=2$ Johnson graph $\G(n, 2, \ell = 1)$ for $n \geq 5$ has a diagonal form with:
\begin{center}
\begin{tabular}{ |c|c|c||c| } 
 \hline
 entry & with multiplicity & if $\ldots$ & or in terms of eigenvalues $\ldots$ \\ 
 \hline
$2(n-1)$ & ${n \choose 2} - 2n + 1$ & always & $[e_2]^{\mu_2 - \mu_1}$ \\

$2(n - 1)n$ & $n - 2$ & always & $[e_1 e_2]^{\mu_1 - 1}$ \\
$1$ & $n - 2$ & always & multiplicity $\mu_1-1$ \\

$1$ & $1$ & $n \equiv 1 \mod 2$ & \\
$4$ & $1$ & $n \equiv 1 \mod 2$ & \\

$2$ & $2$ & $n \equiv 0 \mod 2$ & \\

$0$ & $1$ & always & $[0]^1$ \\
 \hline
\end{tabular}
\end{center}
\end{theorem}
\begin{proof}
In this case, the $\M_s$ matrices (with $\lambda = (n-2)2)$ are
\[ \M_0 = \begin{bmatrix} 0 & 2 & 0 \\ 0 & -n & 2 \\ 0 & 0 & -2n + 2 \end{bmatrix} 
    \hspace{0.1in}
 \M_1 = \begin{bmatrix} -n & 1 \\ 0 & -2n + 2 \end{bmatrix} 
    \hspace{0.1in}
 \M_2 = \begin{bmatrix} -2n + 2 \end{bmatrix}. \]

We use the standard fact that the Smith invariant factors of an integer matrix can be obtained as the pairwise consecutive ratios of the greatest common divisors of the $i \times i$ minors (i.e. determinants of submatrices) of the given matrix \cite{smith1861}.  Let $\GM(s, i)$ be the gcd of the $i \times i$ minors of $\M_s$.

As there are only two non-zero diagonals in each $\M_s$ by Corollary~\ref{c:johnson}, each $i \times i$ minor is simply a product of $i$ of entries from $\M_s$ where no two entries can appear in the same row or column.  These entries consist of integers from the superdiagonal, that we call {\bf degree zero}, as well as the eigenvalues from the diagonal which are linear functions of $n$.

Since the greatest common divisor (gcd) function is associative, we can take the gcd $X$ of the degree zero minors first, and then compare with expressions of $n$ that are not already divisible by $X$.  In this way, we find that:
\begin{itemize}
    \item $\GM(0, 1) = \gcd(2, -n)$
    \item $\GM(0, 2) = \gcd(4, 2(n-1)n) = 4$
    \item $\GM(0, 3) = \gcd(0, 0) = 0$

    \item $\GM(1, 1) = \gcd(1, 0) = 1$
    \item $\GM(1, 2) = \gcd(0, 2(n-1)n) = 2 (n - 1) n$

    \item $\GM(2, 1) = \gcd(0, -2 (n - 1)) = 2(n-1)$
\end{itemize}
Treating $n$ with cases to resolve ambiguity and taking pairwise ratios of these with multiplicities from Theorem~\ref{t:main} yields the results stated in the table.
\end{proof}

This theorem agrees with \cite[Corollary 9.1]{line}, where the authors study critical groups of line graphs.  As far as we know, the following $k=3$ result is new.

\begin{theorem}
The classical Laplacian of the $k=3$ Johnson graph $\G(n, 3, \ell = 2)$ for $n \geq 7$ has a diagonal form with:
\begin{center}
\begin{tabular}{ |c|c|c||c| } 
 \hline
 entry & with multiplicity & if $\ldots$ & or in terms of eigenvalues $\ldots$ \\
 \hline
 $3(n-2)$ & ${n \choose 3} - 2 {n \choose 2} + n$ & always & $[e_3]^{\mu_3 - \mu_2}$ \\

 $6(n-2)(n-1)$ & ${n \choose 2} - 2n + 1$ & always & $[e_2 e_3]^{\mu_2 - \mu_1}$ \\
$1$ & ${n \choose 2} - 2n + 2$ & always & multiplicity $\mu_2 - \mu_1$ \\

$6 (n - 2)(n - 1)n$ & $(n-2)$ & $n \equiv 1 \mod 2$ & $[e_1 e_2 e_3]^{\mu_1 - 1}$ \\
$1$ & $2(n-2)$ & $n \equiv 1 \mod 2$ & multiplicity $2\mu_1 - 2$ \\

$\frac{3}{2} (n - 2)(n - 1)n$ & $(n-2)$ & $n \equiv 0 \mod 2$ & $[\frac{1}{4} e_1 e_2 e_3]^{\mu_1 - 1}$ \\
$2$ & $2(n-2)$ & $n \equiv 0 \mod 2$ & multiplicity $2\mu_1 - 2$ \\

$36$ & $1$ & $n \equiv 2 \mod 3$ & \\
$1$ & $1$ & $n \equiv 2 \mod 3$ & \\

$12$ & $1$ & $n \not\equiv 2 \mod 3$ & \\
$3$ & $1$ & $n \not\equiv 2 \mod 3$ & \\

$0$ & $1$ & always & $[0]^1$ \\
 \hline
\end{tabular}
\end{center}
\end{theorem}
\begin{proof}
Here, we have (for $\lambda = (n-3)3$)
\[ \M_0 = \begin{bmatrix} 0 & 3 & 0 & 0 \\ 0 & -n & 4 & 0 \\ 0 & 0 & -2n + 2 & 3 \\ 0 & 0 & 0 & -3n + 6 \end{bmatrix} 
    \hspace{0.1in}
\M_1 = \begin{bmatrix} -n & 2 & 0 \\ 0 & -2n + 2 & 2 \\ 0 & 0 & -3n + 6 \end{bmatrix} \]
\[ \M_2 = \begin{bmatrix} -2n + 2 & 1 \\ 0 & -3n + 6 \end{bmatrix} 
    \hspace{0.1in}
\M_3 = \begin{bmatrix} -3n + 6 \end{bmatrix}, \]
and we proceed as in the previous proof.  We obtain
\begin{itemize}
    \item $\GM(0, 1) = \gcd(1, 0) = 1$
    \item $\GM(0, 2) = \gcd(3, 2(n-1)n)$
    \item $\GM(0, 3) = \gcd(36, 18(n - 2) (n - 1), -6 (n - 2) (n - 1) n)$
    \item $\GM(0, 4) = \gcd(0, 0) = 0$

    \item $\GM(1, 1) = \gcd(2, -n, -3(n-2))$
    \item $\GM(1, 2) = \gcd(4, 2 (n - 1) n, -2 n, 3 (n - 2) n, -6 (n - 2), 6 (n - 2) (n - 1) )$
    \item $\GM(1, 3) = \gcd(0, -6(n-2)(n-1)n)$

    \item $\GM(2, 1) = \gcd(1, 0) = 1$
    \item $\GM(2, 2) = \gcd(0, 6(n-2)(n-1))$

    \item $\GM(3, 1) = \gcd(0, -3 (m - 2))$
\end{itemize}
Here, we have cases based on $n$ mod $3$ to resolve $\GM(0,2)$ and based on $n$ mod $2$ to resolve $\GM(1,1)$, from which we may then conclude the stated results.
\end{proof}

\subsubsection{Johnson adjacency matrix}

\begin{theorem}
The adjacency matrix of the $k=2$ Johnson graph $\G(n, 2, \ell = 1)$ for $n \geq 5$ has a diagonal form with:
\begin{center}
\begin{tabular}{ |c|c|c||c| } 
 \hline
 entry & with multiplicity & if $\ldots$ & or in terms of eigenvalues $\ldots$ \\ 
 \hline
$2$ & $(n-2)(n-3)/2$ & always &  $[e_2]^{2\mu_{0} - \mu_{1} + \mu_{2}}$\\
$(n - 2)(n - 4)$ & $1$ & always & $[\frac{e_{0}e_{1}}{2}]^{\mu_{0}}$ \\
$2(n-4)$ & $n - 2$ & always & $[2e_{1}]^{\mu_{1}-\mu_{0}}$ \\
$1$ & $n-2$ & always & multiplicity $\mu_1 - \mu_0$ \\
 \hline
\end{tabular}
\end{center}
\end{theorem}
\begin{proof} 
 Using Lemma \ref{l:johnson_M} to compute the $\M_s$ matrices (with $\lambda = 0$), we get
\begin{align*}
\M_0 &= \begin{bmatrix} 2(n-2) & 2 & 0 \\ 0 & n-4 & 2 \\ 0 & 0 & -2 \end{bmatrix} \sim \begin{bmatrix} 2(n-2) & 2 & 0 \\ 0 & n-4 & 0 \\ 0 & 0 & 2 \end{bmatrix}\\
 \M_1 &= \begin{bmatrix} n-4 & 1 \\ 0 & -2 \end{bmatrix} \sim \begin{bmatrix} 1 & 0 \\ 0 & 2(n-4) \end{bmatrix}\\
 \M_2 &= \begin{bmatrix} -2 \end{bmatrix}.
 \end{align*}

Only $\M_{0}$ requires any analysis, since the Smith groups of $\M_1$ and $\M_2$ are evident.
\begin{itemize}
    \item $\GM(0, 1) = \gcd(2, n)$
    \item $\GM(0, 2) = \gcd(4(n-2), 4, 2(n-4), 2(n-2)(n-4))$
    \item $\GM(0, 3) =  4(n-2)(n-4)$
\end{itemize}
Even though the gcds above depend on the parity of $n$, in both cases the Smith group turns out to be isomorphic to
\[
S(\A_{n,2,1}) \cong \left ( \Z/2\Z \right )^{(n-2)(n-3)/2} \oplus \Z/(n-2)(n-4)\Z \oplus \left ( \Z/2(n-4)\Z \right )^{n-2}.
\]
\end{proof}

This theorem agrees with the result \cite[Theorem SNF3]{brouwer:prank}.  We believe the next result is a new one.

\begin{theorem}
The adjacency matrix of the $k=3$ Johnson graph $\G(n, 3, \ell = 2)$ for $n \geq 7$ has a diagonal form with:
\begin{center}
\begin{tabular}{ |c|c|c||c| } 
 \hline
 entry & with multiplicity & if $\ldots$ & or in terms of eigenvalues $\ldots$ \\ 
 \hline
$3(n-7)$ & $\binom{n}{2} - 2n + 1$ & always & $[e_{2}e_{3}]^{\mu_{2} - \mu_{1}}$\\
$3(2n-9)(n-7)$ & $n-2$ & always & $[e_{1}e_{2}e_{3}]^{\mu_{1}-\mu_{0}}$ \\
$3$ & $\binom{n}{3} - 2\binom{n}{2} + n + 1$ & always & $[e_3]^{\mu_3 - \mu_2 + \mu_0}$ \\
$3(n-3)(n-7)(2n-9)/X, X$ & $1$ & always & $[\frac{1}{X}e_{0}e_{1}e_{2}]^{\mu_{0}}, [X]^{\mu_{0}}$ \\
$1$ & $\binom{n}{2}-2$ & always & $\mu_2 + \mu_1 - \mu_0$ \\
 \hline
\end{tabular}
\end{center}
where 
\[
X = \gcd(3(n-3)(2n-9), (n-7)(n-3)(2n-9), 12, 2n(n-7), 3(n-7)).
\]
\end{theorem}
\begin{proof} 
 We again use Lemma \ref{l:johnson_M} to calculate the $\M_s$ matrices (with $\lambda = 0$).  Instead of going straight to gcds of minors, we try playing with some integral row/column operations first.  It turns out that we can get all but $\M_{0}$ directly into diagonal form.
\begin{align*}
\M_0 &= \begin{bmatrix}
   3(n-3) & 3& 0& 0\\
    0& 2n-9& 4& 0\\
    0& 0& n-7& 3\\
    0& 0& 0& -3
\end{bmatrix} \sim \begin{bmatrix}
   (n-3)(2n-9) & 2n& 4& 0\\
    0& 3& 0& 0\\
    0& 0& n-7& 0\\
    0& 0& 0& 3
\end{bmatrix}\\
\M_1 &= \begin{bmatrix} 2n-9 & 2 & 0 \\ 0 & n-7 & 2 \\ 0 & 0 & -3 \end{bmatrix} \sim \begin{bmatrix} 1 & 0 & 0 \\ 0 & 1 & 0 \\ 0 & 0 & 3(2n-9)(n-7) \end{bmatrix}\\
 \M_2 &= \begin{bmatrix} n-7 & 1 \\ 0 & -3 \end{bmatrix} \sim \begin{bmatrix} 1 & 0 \\ 0 & 3(n-7) \end{bmatrix}\\
 \M_3 &= \begin{bmatrix} -3 \end{bmatrix}.
 \end{align*}

If we set
\[
\M_{0^{\prime}} = \begin{bmatrix}
   (n-3)(2n-9) & 2n& 4\\
    0& 3& 0\\
    0& 0& n-7
\end{bmatrix},
\] then we have
\begin{align*}
S(\A_{n,3,2}) &\cong S(\M_{0})^{\mu_{0}-\mu_{-1}} \oplus S(\M_{1})^{\mu_{1}-\mu_{0}} \oplus S(\M_{2})^{\mu_{2}-\mu_{1}} \oplus S(\M_{3})^{\mu_{3}-\mu_{2}}\\ 
&\cong S\left ( \M_{0^{\prime}} \right ) \oplus \Z/3\Z 
\oplus \left ( \Z/3(2n-9)(n-7)\Z \right )^{n-2} \\
&\qquad \oplus \left ( \Z/3(n-7)\Z \right )^{\binom{n}{2} - 2n + 1} \oplus \left ( \Z/3\Z \right )^{\binom{n}{3} - 2\binom{n}{2} + n}. 
\end{align*}

Denoting by $\GM(0^{\prime},i)$ the gcd of the $i \times i$ minors of $\begin{bmatrix}
   (n-3)(2n-9) & 2n& 4\\
    0& 3& 0\\
    0& 0& n-7
\end{bmatrix}$, we observe
\begin{itemize}
    \item $\GM(0^{\prime}, 1) = 1$
    \item $\GM(0^{\prime}, 2) = \gcd(3(n-3)(2n-9), (n-7)(n-3)(2n-9), 12, 2n(n-7), 3(n-7))$
    \item $\GM(0^{\prime}, 3) = 3(n-3)(n-7)(2n-9)$.
\end{itemize}
Setting $X = \GM(0^{\prime}, 2)$ we therefore have
\[
\M_{0^{\prime}} \sim \begin{bmatrix}
   1 & 0& 0\\
    0& X& 0\\
    0& 0& 3(n-3)(n-7)(2n-9)/X
\end{bmatrix}.
\]
The theorem follows.  
\end{proof}
We note that the value of $X$ in the above theorem depends only on $n \pmod{12}$.

\subsection{A non-square subset inclusion matrix}

As a demonstration of our results in the case when $\A$ is not square, consider the subset inclusion matrix with $(\kr,\kc) = (2, 3)$ and $\ell = 1$.

\begin{theorem}
The subset intersection matrix $\A_{n, 2, 3, 1}$ for $n \geq 5$ has a diagonal form with:
\begin{center}
\begin{tabular}{ |c|c|c| } 
 \hline
 entry & with multiplicity & if $\ldots$ \\
 \hline
 $2$ & ${n \choose 2} - 2n + 1$ & always \\
 $2(n - 6)$ & $n - 2$ & always \\
 $1$ & $n - 1$ & always \\
 $(n-3)(n-6)$ & $1$ & if $3 | n$ \\
 $6$ & $1$ & if $3 | n$ \\
 $3(n-3)(n-6)$ & $1$ & if $3 \not | n$ \\
 $2$ & $1$ & if $3 \not | n$ \\
 \hline
\end{tabular}
\end{center}
\end{theorem}
\begin{proof}
The $\M_s$ matrices are ($\lambda = 0$):
\[ \M_0 = \begin{bmatrix} (n-2)(n-3) & 2(n-3) & 2 & 0 \\ 0 & \frac{1}{2}(n-3)(n-6) & 2(n-5) & 3 \\ 0 & 0 & -2(n-4) & -6 \end{bmatrix}  \]
\[ \M_1 = \begin{bmatrix} \frac{1}{2}(n-3)(n-6) & (n-5) & 1 \\ 0 & -2(n-4) & -4 \end{bmatrix}  \]
\[ \M_2 = \begin{bmatrix} -2(n-4) & -2 \end{bmatrix}  \]
and we have
\begin{itemize} 
    \item $\GM(0,1) = 1$
    \item $\GM(0,3) = 6(n - 3)(n - 6)$
    \item $\GM(1,1) = 1$
    \item $\GM(1,2) = \gcd( -(n - 3)(n - 4)(n - 6), -2(n - 3)(n - 6), -2(n - 6) ) = 2(n-6)$
    \item $\GM(2,1) = 2$.
\end{itemize}

Examining the minors in $\GM(0,2)$, we note that all of them are divisible by $2$ and that $6$ is itself a minor, so $\GM(0,2)$ must be $2$ or $6$.
If $n$ is divisible by $3$, we get $6$, otherwise $2$.  Therefore, we obtain the result.
\end{proof}

Note that this agrees with (and also refines for the ``$\kc=3$'' case) \cite[Theorem 4.1]{ds}.

\subsection{Kneser graphs}

When $\ell = 0$, we obtain the following simple form for the $f_i(j)$.

\begin{lemma}\label{l:f_kneser}
If $\ell = 0$, we have
\[ f_{i}(j) = (-1)^i {n - k - j \choose k - j}. \]
\end{lemma}
\begin{proof}
Referring back to Definition~\ref{d:cf}, we have
\[ c_i = c_i(j; n, k, 0) = \begin{cases}
    {n-k-j \choose k-j} & \text{ if $i = 0$ } \\ 
    0 & \text{ otherwise. } \\ 
\end{cases} \]
Thus, $f_i(j; n, k, 0) = \sum_{v=0}^i (-1)^{i+v} {i \choose v} c_v(j; n, k, 0) = (-1)^{i+0} {i \choose 0} c_0(j; n, k, 0)$
as claimed.
\end{proof}

\subsubsection{Kneser adjacency matrix}

\begin{theorem}
The adjacency matrix of the Kneser graph $\Gamma(n, k, \ell = 0)$ for $n \geq 3k-1$ has a diagonal form with:
\begin{center}
\begin{tabular}{ |c|c||c| } 
 \hline
 entries & with multiplicity & or in terms of eigenvalues $\ldots$ \\ 
 \hline
 $\binom{n-k-j}{k-j}$  & ${n \choose j} - {n \choose j-1}$ & $[e_j]^{\mu_j} \quad \mbox{ for $0 \leq j \leq k$}$ \\
 \hline
\end{tabular}
\end{center}
\end{theorem}
\begin{proof}
    By Lemma \ref{l:f_kneser}, the matrices $\M_{s}$ (with $\lambda = 0$) satisfy 
    \[
    \M_s(i,j) = (-1)^i {j - s \choose i - s} {n - k - j \choose k - j}.
    \]
    Observe that each of the diagonal entries $M_{s}(j,j)$ divides each entry $M_{s}(i,j)$ above it ($i<j$).  Thus every $\M_{s}$, and therefore $\A_{n,k,0}$, is unimodularly equivalent to the diagonal matrix of its eigenvalues.  
\end{proof}
The fact that the diagonal entries of a  diagonal form of $\A_{n,k,0}$ are equal (including multiplicities) to the eigenvalues of $\A_{n,k,0}$ has been described as ``miraculous;'' see \cite[Cor. 9.4.4]{brouwer-vanmaldeghem}.  This result also follows from Wilson's original diagonal form \cite[Theorem 2]{wilson_classic} of the unrestricted inclusion matrices, since a $k$-subset being \textit{disjoint} from another $k$-subset is the same thing as being \textit{included} in the size-$(n-k)$ complement.

\subsubsection{Kneser Laplacian matrix}

From Lemma \ref{l:f_kneser} the entries of the $\M_s$ matrices (with $\lambda = d = \binom{n-k}{k}$) are
\[ \M_s(i,j) = (-1)^i {j - s \choose i - s} {n - k - j \choose k - j} - \delta_{i,j} {n-k \choose k} \]
by Theorem~\ref{t:main}.  The minors here are more complicated than in Johnson case.
However, the diagonal entries (eigenvalues)
\[ \M_s(i,i) = (-1)^i {n - k - i \choose k - i} - {n-k \choose k}  \]
often seem to factor nicely; for $k = 2$ we obtain
\[ 0, -\frac{1}{2}(n-3)n, -\frac{1}{2}(n-1)(n-4) \]
while for $k = 3$, we have
\[ 0, -\frac{1}{6}(n-4)(n-5)n, -\frac{1}{6}(n-1)(n-5)(n-6), -\frac{1}{6}(n^2 - 10n + 27)(n-2). \]

\begin{theorem}\label{t:k2k}
The classical Laplacian of the $k=2$ Kneser graph $\G(n, 2, \ell = 0)$ for $n \geq 5$ has a diagonal form with:
\begin{center}
\begin{tabular}{ |c|c|c||c| } 
 \hline
 entry & with multiplicity & if $\ldots$ & or in terms of eigenvalues $\ldots$ \\ 
 \hline
 $\frac{1}{2}(n - 1)(n - 4)$ & ${n \choose 2} - 2n + 1$ & always & $[e_2]^{\mu_2 - \mu_1}$ \\
 $\frac{1}{4}(n - 1)(n - 3)(n - 4)n$ & $n - 2$ & always & $[e_1 e_2]^{\mu_1 - 1}$ \\
 $1$ & $n - 1$ & always & $\text{multiplicity } \mu_1$ \\
${n-3 \choose 2} = \frac{1}{2}(n-3)(n-4)$ & $1$ & always & multiplicity $1$ \\
$0$ & $1$ & always & $[0]^1$ \\
 \hline
\end{tabular}
\end{center}
\end{theorem}
\begin{proof}
In this case, the $\M_s$ matrices are (for $\lambda = \binom{n-2}{2}$)
\[ \M_0 = \begin{bmatrix} 0 & n-3 & 1 \\ 0 & -{n-2 \choose 2} - (n-3) & -2 \\ 0 & 0 & -{n-2 \choose 2} + 1 \end{bmatrix} 
= \begin{bmatrix} 0 & n-3 & 1 \\ 0 & -\frac{1}{2}(n-3)n & -2 \\ 0 & 0 & -\frac{1}{2}(n-4)(n-1) \end{bmatrix}  \]
with
\[ \M_1 = \begin{bmatrix} -\frac{1}{2}(n-3)n & -1 \\ 0 & -\frac{1}{2}(n-4)(n-1) \end{bmatrix}  
\hspace{0.1in}
\M_2 = \begin{bmatrix} -\frac{1}{2}(n-4)(n-1) \end{bmatrix}. \]

We have 
\begin{itemize} 
    \item $\GM(0,1) = 1$
    \item $\GM(0,2) = \gcd(\frac{1}{2}(n - 3)(n - 4), -\frac{1}{2}(n - 1)(n - 3)(n - 4), \frac{1}{4}(n - 1)(n - 3)(n - 4)n) = {n-3 \choose 2}$
    \item $\GM(0,3) = 0$
    \item $\GM(1,1) = 1$
    \item $\GM(1,2) = \frac{1}{4}(n - 1)(n - 3)(n - 4)n$
    \item $\GM(2,1) = -\frac{1}{2}(n - 1)(n - 4)$
\end{itemize}

Remarkably, no cases were needed!
\end{proof}

Theorem~\ref{t:k2k} agrees with the main result in \cite{duceysinhill}.  As far as we know, the next result is new.

\begin{theorem}\label{t:kk3}
The classical Laplacian of the $k=3$ Kneser graph $\G(n, 3, \ell = 0)$ for $n \geq 7$ has a diagonal form with:
\begin{center}
\begin{tabular}{ |c|c||c| } 
 \hline
 entries & with multiplicity & or in terms of eigenvalues $\ldots$ \\ 
 \hline
 $\frac{1}{6}(n^2 - 10n + 27)(n-2)$ & ${n \choose 3} - 2{n \choose 2} + n$ & $[e_3]^{\mu_3 - \mu_2}$ \\
 $\frac{1}{36}(n^2 - 10n + 27)(n - 1)(n - 2)(n - 5)(n - 6)$ & ${n \choose 2} - 2n + 1$ & $[e_2 e_3]^{\mu_2 - \mu_1}$ \\
 $1$ & ${n \choose 2} - n$ & multiplicity $\mu_2$ \\

 $\frac{\frac{1}{216}(n^2 - 10n + 27)(n - 1)(n - 2)(n - 4)(n - 5)^2 (n - 6)n}{X}$, $X$ & $n - 2$ & $\left[\frac{1}{X} e_1 e_2 e_3\right]^{\mu_1 - 1}$, $[X]^{\mu_1 - 1}$ \\
 $\frac{\frac{1}{36}(n^2 - 10n + 27)(n - 4)(n - 5)^2(n - 6)}{Y}$, $Y$ & $1$ & multiplicity $1$ (twice)\\
 $0$ & $1$ & $[0]^1$ \\
 \hline
\end{tabular}
\end{center}
where 
\[ X = \gcd\left( \frac{(n - 4)(n - 5)n}{3}, \frac{(n^2 - 7n + 18)(n - 5)}{6}, \frac{(n^2 - 10n + 27)(n - 2)(n - 4)(n - 5)n}{36} \right) \]
and
\[ Y = \gcd\left( n-5,\ \frac{3(n-4)(n-5)}{2},\ \frac{(n-1)(n-3)(n-5)}{3} \right). \]
\end{theorem}

\begin{proof}[{\it Proof sketch}]
In this case, the $\M_s$ matrices are (with $\lambda = \binom{n-3}{3}$)
\[ \M_0 = \begin{bmatrix} 0 & \frac{1}{2}(n-4)(n-5) & n-5 & 1 \\ 0 & -\frac{1}{6}(n-4)(n-5)n & -2(n-5) & -3 \\ 0 & 0 & -\frac{1}{6}(n-1)(n-5)(n-6) & 3 \\ 0 & 0 & 0 & -\frac{1}{6}(n^2 - 10n + 27)(n-2) \end{bmatrix}  \]
\[ \M_1 = \begin{bmatrix} -\frac{1}{6}(n-4)(n-5)n & -(n-5) & -1 \\ 0 & -\frac{1}{6}(n-1)(n-5)(n-6) & 2 \\ 0 & 0 & -\frac{1}{6}(n^2 - 10n + 27)(n-2) \end{bmatrix}  \]
\[ \M_2 = \begin{bmatrix} -\frac{1}{6}(n-1)(n-5)(n-6) & 1 \\ 0 & -\frac{1}{6}(n^2 - 10n + 27)(n-2) \end{bmatrix}  \]
\[ \M_3 = \begin{bmatrix} -\frac{1}{6}(n^2 - 10n + 27)(n-2) \end{bmatrix}, \]
and most of the minor gcd computations are straightforward:
\begin{itemize} 
    \item $\GM(0,1) = \GM(1,1) = \GM(2,1) = 1$
    \item $\GM(0,3) = \frac{1}{36}(n^2 - 10n + 27)(n - 4)(n - 5)^2(n - 6)$
    \item $\GM(0,4) = 0$
    \item $\GM(1,3) = \frac{1}{216}(n^2 - 10n + 27)(n - 1)(n - 2)(n - 4)(n - 5)^2 (n - 6)n$
    \item $\GM(2,2) = \frac{1}{36}(n^2 - 10n + 27)(n - 1)(n - 2)(n - 5)(n - 6)$
    \item $\GM(3,1) = \frac{1}{6}(n^2 - 10n + 27)(n-2)$.
\end{itemize}
However, it remains to work out values for $X = \GM(1,2)$ and $Y = \GM(0,2)$.

For $X$, the complete list of minors (up to sign, obtained using Sage software \cite{sagemath}) is
\[ \begin{aligned}[t] 
 x_1 :=\ & (n - 4)(n - 5)n/3, \\[-.3em]
 x_2 :=\ & (n^2 - 7n + 18)(n - 5)/6, \\[-.3em]
 x_3 :=\ & (n^2 - 10n + 27)(n - 2)(n - 4)(n - 5)n/36, \\[-.3em]
 & (n - 1)(n - 4)(n - 5)^2(n - 6)n/36, \\[-.3em]
 & (n^2 - 10n + 27)(n - 2)(n - 5)/6, \\[-.3em]
 & (n^2 - 10n + 27)(n - 1)(n - 2)(n - 5)(n - 6)/36 \\[-.3em]
\end{aligned} \]
and we claim that the last three are redundant in the gcd computation.  For example, the fourth row $(n - 1)(n - 4)(n - 5)^2(n - 6)n/36$ can already be written as an (integral!) linear combination 
\[ -(n+5)x_1 + \frac{n(n-4)(n-5)}{6} x_2 \]
of the first and second rows.  Thus by B\'ezout's identity we may ignore it in the computation of the gcd of these minors.  One may obtain similar integer linear combinations for the fifth and sixth rows, resulting in the expression for $X$ given in the statement.

Similarly for $Y$, one may consider the complete list of nonzero minors
\[ \begin{aligned}[t] 
 y_1 :=\ & (n - 5), \\[-.3em]
 y_2 :=\ & (3/2)(n - 4)(n - 5), \\[-.3em]
 & (1/6)(n - 4)(n - 5)^2(n - 6), \\[-.3em]
 & (1/6)(n - 4)(n - 5)(n - 9), \\[-.3em]
 & (1/12)(n - 1)(n - 4)(n - 5)^2(n - 6), \\[-.3em]
 & (1/6)(n^2 - 7n + 24)(n - 5), \\[-.3em]
 & (1/12)(n^2 - 10n + 27)(n - 2)(n - 4)(n - 5), \\[-.3em]
 & (1/6)(n^2 - 10n + 27)(n - 2)(n - 5), \\[-.3em]
 & (1/36)(n - 1)(n - 4)(n - 5)^2(n - 6)n, \\[-.3em]
 & (1/2)(n - 4)(n - 5)n, \\[-.3em]
 & (1/2)(n^2 - 7n + 18)(n - 5), \\[-.3em]
 & (1/36)(n^2 - 10n + 27)(n - 2)(n - 4)(n - 5)n, \\[-.3em]
 & (1/3)(n^2 - 10n + 27)(n - 2)(n - 5), \\[-.3em]
 & (1/36)(n^2 - 10n + 27)(n - 1)(n - 2)(n - 5)(n - 6), \\[-.3em]
\end{aligned} \]
and observe that each may be expressed as an integral linear combination of the first two rows, together with $y_3 := (1/3)(n-1)(n-3)(n-5)$.  Some helpful formulas in this pursuit are:
\[ n^2 - 7n + 18 = n(n-1)-6(n-3), \]
\[ n^2 - 7n + 24 = n(n-1)-6(n-4), \]
\[ n^2 - 10n + 27 = (n-1)(n-3)-6(n-4). \]
Taking pairwise ratios of the resulting gcds yields the results stated in the table.
\end{proof}

\bigskip
While we have tried to present our results in a reasonably explicit style, that reflects existing theorems in the literature, it is difficult to avoid confronting a meta-mathematical issue (e.g. \cite{wilf}) regarding ``What (precisely) constitutes a satisfactory answer for these diagonal form problems?''  

In Theorem~\ref{t:kk3}, for example, there does not seem to be a canonical way to choose the formulas from the list of minors of $\M_s$ to which we apply our integer gcd operation.  In fact, we did not verify that two gcd computations with three terms each {\em is} the minimal form of a correct statement.  Yet, in some sense, we feel that the $\M_s$ matrices themselves are already the real ``answer'' and that any case statements on $n$ or gcd computations we employ are more in the direction of ``fiddly details'' reflecting a subjective expectation about the form of the answer, rather than anything intrinsic to the mathematics of the question itself.

By contrast, we have tried to highlight in the last column of each table how our diagonal forms may be written in terms of the eigenvalues, reflecting a deeper relationship between them.  These examples seem to suggest the existence of a unified formula, independent from both $n$ and $\ell$, that realizes the Smith group directly in terms of the eigenvalues and their multiplicities.

\appendix
\section*{An appendix on some super-standard combinatorics}
\setcounter{section}{5}  
\setcounter{theorem}{0}  
\setlength{\arraycolsep}{1.5pt}
\setcounter{MaxMatrixCols}{99}

The careful reader will have noticed that there is some choice in our construction of the $E_s$ matrices that diagonalize the $\W_{i,j}$ blocks in Theorem~\ref{t:bier-std} at the point where we appeal to Wilson's Proposition 2 in order to adjoin rows in the induction.  Although it is not logically necessary for any of our results, we offer the following outline towards a canonical choice for $E_s$ for those interested in implementing these matrices in a standard fashion or in comparing with the non-recursive situation that Bier found vis-\`a-vis $\P$ and $\W_{i,j}^{\mathsf{classic}}$.

\begin{definition}
We say that a set $\beta = \{b_1, b_2, \ldots, b_k\} \subseteq \{1, 2, \ldots, n\}$ is {\bf super-standard} if 
\[ 2i \leq b_i < (n-2k) + 2i \ \ \text{ for all $i = 1, 2, \ldots, k$ } \]
where the $b_i$ form an ordered labeling of the elements (so $b_1 < b_2 < \ldots < b_k$).
\end{definition}

If we visualize our sorted list of entries, graphically, as an increasing selection of columns for each row from $1$ to $k$ in a table, it may be easier to see how this definition relates to the earlier ones we've used.  For example, with $n = 12$ and $k = 3$, we obtain
\[ \scalemath{0.8}{
\begin{pmatrix}
    1 & 2 & 3 & 4 & 5 & 6 & 7 & 8 & 9 & 10 & 11 & 12 \\
    1 & 2 & 3 & 4 & 5 & 6 & 7 & 8 & 9 & 10 & 11 & 12 \\
    1 & 2 & 3 & 4 & 5 & 6 & 7 & 8 & 9 & 10 & 11 & 12 \\
\end{pmatrix} \ \ 
\begin{pmatrix}
    . & 2 & 3 & 4 & 5 & 6 & 7 & 8 & 9 & 10 & 11 & 12 \\
    . & . & . & 4 & 5 & 6 & 7 & 8 & 9 & 10 & 11 & 12 \\
    . & . & . & . & . & 6 & 7 & 8 & 9 & 10 & 11 & 12 \\
\end{pmatrix} \ \ 
\begin{pmatrix}
    . & 2 & 3 & 4 & 5 & 6 & 7 & . & . & .  & .  & . \\
    . & . & . & 4 & 5 & 6 & 7 & 8 & 9 & .  & .  & . \\
    . & . & . & . & . & 6 & 7 & 8 & 9 & 10 & 11 & . \\
\end{pmatrix}
} \]
as our tables for the unrestricted, standard, and super-standard subsets, respectively.  The selection $\{2, 10, 11\}$ is a valid subset (of increasing columns for each of the three rows) in the first two tables, but not for the third.  By definition, there will always be $n-2k$ entries in each row of the super-standard table.

\begin{definition}
Suppose $\widetilde\W_{i,j}(n)$ is the inclusion matrix whose rows are indexed by {\it super}-standard $i$-subsets of $\{1, \ldots, n\}$, and whose columns are indexed by {\it standard} $j$-subsets, with entries $\widetilde\W_{i,j}(\alpha, B) = 1$ if $\alpha \subseteq B$ and $\widetilde\W_{i,j}(\alpha, B) = 0$ otherwise.  For $0 \leq i, j \leq n$, let $\widetilde{\P}_{i,j}(n)$ be the stacked matrix $\bigcup_{s=0}^{i}\widetilde\W_{s,j}(n)$.
\end{definition}

Based on computational evidence and the results in this appendix, we believe the following.

\begin{conjecture}\label{c:sso}
Let $0 \leq i \leq \frac{n+1}{3}$ and $0 \leq j \leq \frac{n+1}{3}$.  Then, $\widetilde{\P}_{i,j}$ has dimensions $\mu_i \times \mu_j$, index $1$, and full rank.  Consequently, $\widetilde\P_{k,k}$ is unimodular for all $n \geq 3k - 1$.
\end{conjecture}

\begin{remark}
The $j$ parameter for $\widetilde\P$ is not unrestricted:  recall that when $j > \frac{n}{2}$ there are no $j$-standard subsets of $n$ at all.  Moreover, we empirically find failures of the full-rank condition for some values of $j$ even when $i$ is within its bounds; for example, $\widetilde\P_{3,4}(9)$ has dimensions $48 \times 42$ but rank $41$. 
\end{remark}

This combinatorics was developed based a ``greedy'' implementation of the construction in Theorem~\ref{t:bier-std}.  Interestingly though, the specific inequalities in the super-standard definition don't really play a direct role in diagonalizing the $\W_{i,j}$.  It turns out that any sub-collection of standard subsets that induce the correct dimensions could have done that!  (Most sub-collections will not have unimodular change-of-basis matrices, however.)

\begin{lemma}\label{l:simpler}
Suppose $i \leq j$.  If $\widetilde\P_{i,i}$ is any matrix with dimensions $\mu_i \times \mu_i$ that encodes the inclusion relation for some collection of ``super-standard'' objects (and similarly for $\widetilde\P_{j,j}$), then
\[ \widetilde\P_{i,i} \ \W_{i,j} =  D_{i,j} \ \widetilde\P_{j,j} \]
where $D$ is the diagonal matrix from Definition~\ref{d:diag_for_w}.
\end{lemma}
\begin{proof}
Consider the $(A, B)$-entry of the matrix product on each side of the equality, where $A$ is a super-standard $(\leq i)$-subset, and $B$ is a standard $j$-subset.  The $(A, B)$-entry of the left side $\widetilde\P_{i,i} \W_{i,j}$ is a count of the standard $i$-subsets $S$ such that $A \subseteq S \subseteq B$.  Since $B$ is standard, and any subset of a standard subset is standard, we may view $S$ as unrestricted.  (The fact that $A$ is ``super-standard'' places no restrictions on $S$.) Letting $s = |A|$, we find that the $(A,B)$-entry is simply the number of these subsets $S$, namely:  ${ j-s \choose i-s }$ if $A \subseteq B$, or $0$ if $A \not \subseteq B$.

We have that $D_{i,j}$ is a diagonal matrix with the same dimensions $\mu_i \times \mu_j$ as $\W_{i,j}$.  To evaluate the right side, we assume that rows and columns are ordered compatibly in the following sense: 
\begin{quote}
The rows of $\widetilde\P_{j,j}$ are indexed by super-standard $(\leq j)$-subsets, which contain the super-standard $(\leq i)$-subsets.  We arrange the $(\leq i)$-subsets first so that the initial rows of $\widetilde\P_{i,i}$ and $\widetilde\P_{j,j}$ agree.
\end{quote}
Then, the $(A, B)$-entry of $D_{i,j} \widetilde\P_{j,j}$ is nominally a sum over super-standard $(\leq j)$-subsets $T$ where $T \subseteq B$.  But as $D_{i,j}$ only includes non-zero entries along the diagonal, it really it singles out one entry:  $T$ corresponds to the row of $\widetilde\P_{j,j}$ that was indexed by $A$ on the left side of the equality.  Since $A$ is a super-standard $(\leq i)$-subset, it is also a super-standard $(\leq j)$-subset and, by our compatibility assumption, these agree.

Thus, if $A \subseteq B$ then the entry of $\widetilde\P_{j,j}$ corresponding to row $A$ column $B$ will be $1$ and after multiplying by $D_{i,j}$ we get the same binomial coefficient formula on the right side as we did on the left side; otherwise, we get $0$.
\end{proof}

So, {\bf assuming that Conjecture~\ref{c:sso} is true (and it can be easily checked by any code that uses it), Lemma~\ref{l:simpler} shows we may take $\widetilde\P_{s,s}$ for the $E_s$ matrices} in the earlier sections of this paper.  

Turning towards actually proving the conjecture, we would like to break into smaller blocks based on the ``add or remove $n$'' bijection.  Unfortunately, this immediately fails as none of the super-standard subsets contain $n$.  Instead, we take the following approach.

\begin{definition}
We say that a standard subset is {\bf on the boundary} if it contains any entry $2i$ in sorted position $i$.  These are the leftmost entries from each row of the corresponding graphical tables for standard and super-standard subsets.  Otherwise, we say that the standard subset is {\bf in the interior}.
\end{definition}

As an example, $\{3, 4, 7\}$ is a subset that is on the boundary while $\{4, 5, 8\}$ is in the interior.  Even though both subsets contain the element $4$, it is only a boundary entry for the first subset.

\begin{proposition}
Fix $n$ and $k$.  The super-standard $k$-subsets of $n$ in the interior are in bijection with the set of (all) super-standard $k$-subsets of $n-1$.  Thus, the ``interior block'' obtained by restricting rows and columns of $\widetilde\P_{i,j}(n)$ to subsets in the interior is just $\widetilde\P_{i,j}(n-1)$.  In addition, the ``row boundary/column interior'' block is zero.
\end{proposition}
\begin{proof}
We can simply subtract $1$ from the value of each entry in the subset to pass from the first collection to the second.  This is a reversible operation that preserves subset inclusion.

Also, no subset $R$ from the boundary collection can be contained in a subset $S$ from the interior collection.  To see this, suppose that $e$ is a boundary entry (that is, $e = 2i$ in position $i$) of $R$.  If $R \subseteq S$, there are $i$ entries (including $e$), with values between $2$ and $e$, that would all need to appear among the first $i-1$ positions from $S$, since all of the entries of $S$ at position $i$ and greater must have values larger than $e$, by the interior condition.  This is impossible.
\end{proof}

The ``boundary block'' of $\widetilde\P_{i,j}(n)$ turns out to be a bit more complicated.  We need a reversible projection (the analogue of ``removing $n$'') in order to relate it to a matrix we can already understand via an induction hypothesis.

\begin{definition}\label{d:phi}
Given a standard $k$-subset of $n$ on the boundary, called $S$ say, we define an operation $\varphi(S)$ that removes one entry from $S$.  Look at the graphical table corresponding to $S$, and attempt to remove entries corresponding to the leftmost column, $2$, $4$, $6$, etc. in each row of the graphical table, working from row $1$, forwards to row $k$.  As soon as one of these entries $e$ are found in (the correct sorted position of!) $S$, remove $e$ and subtract $1$ from each of the rest of the entries, returning 
\[ \varphi(S) = \{ s-1 : s \in S \setminus \{e\} \}. \]
\end{definition}

\begin{example}\label{e:phi}
Here is an example from $n = 12$.  We show that the super-standard $3$-subsets of $n$ on the boundary are in bijection with (all of) the super-standard $2$-subsets of $n-1$.  View each table diagram as a {\em collection} of super-standard subsets satisfying conditions on the entries illustrated in the table.  We put a $\partial$ symbol in front of the diagram to mean that we only select subsets on the boundary for our collection, and use $+$ for union.  Then we can perform the following ``diagram calculus'' to verify that $\varphi$ is a bijection.

In the first step, we project the subsets containing $2$ to $k = 2$ by removing it and shifting the remaining entries left by one, while subsets that do not contain $2$ remain to be considered:
\[ \scalemath{0.8}{
\partial 
\begin{pmatrix}
    {\ul 2} & 3 & 4 & 5 & 6 & 7  \\
    &   & 4 & 5 & 6 & 7 & 8 & 9 \\
    &   &   &   & 6 & 7 & 8 & 9 & 10 & 11 \\
\end{pmatrix}
= \begin{pmatrix}
    &   & 3 & 4 & 5 & 6 & 7 & 8 \\
    &   &   &   & 5 & 6 & 7 & 8 & 9 & 10 \\
\end{pmatrix}
+
\partial \begin{pmatrix}
       & 3 & 4 & 5 & 6 & 7  \\
        &   & 4 & 5 & 6 & 7 & 8 & 9 \\
        &   &   &   & 6 & 7 & 8 & 9 & 10 & 11 \\
\end{pmatrix}.
} \]
Notice that when $2$ is present and we remove it, we obtain all the elements (both on the boundary and in the interior) in the projected collection for $k=2$.
Performing the same dichotomy and projection for using the element $4$ (when it appears in the second position), we obtain
\[ \scalemath{0.8}{
\partial \begin{pmatrix}
       & 3 & 4 & 5 & 6 & 7  \\
     &   & {\ul 4} & 5 & 6 & 7 & 8 & 9 \\
        &   &   &   & 6 & 7 & 8 & 9 & 10 & 11 \\
\end{pmatrix}
= 
\begin{pmatrix}
    2 & . & . & . & . & .  \\
      &  &  & 5 & 6 & 7 & 8 & 9 & 10 \\
\end{pmatrix}
+ \partial \begin{pmatrix}
        & 3 & 4 & 5 & 6 & 7  \\
        &   &   & 5 & 6 & 7 & 8 \\
        &   &   &   & 6 & 7 & 8 & 9 & 10 \\
\end{pmatrix}.
} \]
Here, we are using that the selected columns must increase, so if $4$ was present in the subset on the left it must have also had a $3$ in the first row (which becomes a $2$ after the entries are shifted left by one).  Similarly, for $6$,
\[ \scalemath{0.8}{
\partial \begin{pmatrix}
        & 3 & 4 & 5 & 6 & 7  \\
        &   &   & 5 & 6 & 7 & 8 \\
        &   &   &   & {\ul 6} & 7 & 8 & 9 & 10 \\
\end{pmatrix}
= \begin{pmatrix}
        & 2 & 3 & . & . & . & . \\
        &   &   & 4 & .  & . & . & . & . \\
\end{pmatrix}
} \]
so we obtain
\[ \scalemath{0.8}{
\partial 
\begin{pmatrix}
    {\ul 2} & 3 & 4 & 5 & 6 & 7  \\
    &   & 4 & 5 & 6 & 7 & 8 & 9 \\
    &   &   &   & 6 & 7 & 8 & 9 & 10 & 11 \\
\end{pmatrix}
= \begin{pmatrix}
    &   & 3 & 4 & 5 & 6 & 7 & 8 \\
    &   &   &   & 5 & 6 & 7 & 8 & 9 & 10 \\
\end{pmatrix}
+
\begin{pmatrix}
    2 & . & . & . & . & .  \\
      &  &  & 5 & 6 & 7 & 8 & 9 & 10 \\
\end{pmatrix}
+ \begin{pmatrix}
        & 2 & 3 & . & . & . & . \\
        &   &   & 4 & .  & . & . & . & . \\
\end{pmatrix}
= \begin{pmatrix}
    & 2 & 3 & 4 & 5 & 6 & 7 & 8 \\
    &   &   & 4 & 5 & 6 & 7 & 8 & 9 & 10 \\
\end{pmatrix}
} \]
as desired.  Observe that we recover the entire {\em interior} of the image at the first step (which itself is revealed to be isomorphic to the set of (all) super-standard $(k-1)$-subsets in $n-2$), while the subsequent steps eventually recover the {\em boundary} of the image set.
\end{example}

\begin{lemma}\label{l:sso_biject}
The map $\varphi$ is a bijection of the various collections shown below.  For the columns of the boundary block of $\widetilde\P_{i,j}(n)$, we have
\[ \partial\left( \text{super-standard $j$-subsets of $n$} \right)  \xrightarrow{\ \ \varphi \ \ } \left( \text{ super-standard $(j-1)$-subsets of $n-1$ } \right) \]
\[ \partial\left( \text{non super-standard $j$-subsets of $n$} \right)  \xrightarrow{\ \ \varphi \ \ } \left( \text{ non super-standard $(j-1)$-subsets of $n-1$ } \right) \]
and for the rows, we have
\[ \partial\left( \text{super-standard $(\leq i)$-subsets of $n$} \right)  \xrightarrow{\ \ \varphi \ \ } \left( \text{ super-standard $(\leq (i-1))$-subsets of $n-1$ } \right). \]
So, these images of $\varphi$ are a permutation of the indexing elements for $\widetilde\P_{i-1,j-1}(n-1)$.
\end{lemma}
\begin{proof}
The first claim follows in general just as for Example~\ref{e:phi}.  To reverse the map, look for the lowest row containing a boundary entry ($2$, $4$, $6$, etc.), call it $e$ say, and ``inflate'' by shifting every entry to the right by one and then adding a new row just below it that contains $e+2$.  The third claim follows by repeated application of the first claim.  As the $\varphi$ map only operates on the left side of the table diagram, it won't ever remove an entry that witnesses a failure of the super-standard condition.  Consequently, the image in the second claim will still be non super-standard in $n-1$.
\end{proof}

Our bijections are already enough to recover an inductive proof for the enumeration: there are precisely $\mu_i(n) - \mu_{i-1}(n)$ super-standard $i$-subsets of $n$.  
Next, we consider how $\varphi$ affects the inclusion relations themselves.

\begin{lemma}\label{l:ssoequiv}
When we apply $\varphi$ to the rows and columns of the boundary block of $\widetilde\P_{i,j}(n)$, each of the columns indexed by a subset that contains a $2$-entry will map to precisely the same column vector in $\widetilde\P_{i-1,j-1}(n-1)$.  (The other column vectors may gain some additional $1$ entries as we apply $\varphi$.)
\end{lemma}
\begin{proof}
Fix a $(\leq i)$ super-standard subset of $n$, called $R$ say, that indexes some row of the boundary block of $\widetilde\P_{i,j}(n)$.  Let $S$ be a standard $j$-subset of $n$ that indexes a column of the boundary block, and assume that $2 \in S$ (in the first position).

By definition, we know $\varphi(S) = \{ s-1 \in S \setminus \{2\} \}$ and $\varphi(R) = \{ r-1 \in R \setminus \{e\} \}$ for the first $e = 2, 4, 6, \ldots$ that exists in $R$.  If $e = 2$, we clearly have that $R \subseteq S$ if and only if $\varphi(R) \subseteq \varphi(S)$.

So suppose $e > 2$.  Then, $2 \notin R$ to begin with so removing it from $S$ won't change inclusion.  Thus, we have $R \subseteq S$ implies $\varphi(R) \subseteq \varphi(S)$.  Now if $R \not \subseteq S$ then there exists some other element $g \in R$ that is not in $S$.  Choose this so that it has minimal position among such elements.  As long as $g \neq e$, it will survive to show $\varphi(R) \not \subseteq \varphi(S)$, so suppose $g = e$.

In this case, $e-1$ must be the element we selected for the row above $e$ in $R$ and, by the minimality of $g$, we must have $e-1 \in S$.  Similarly, one of $\{e-3, e-2\}$ must be the element we selected for the row above $e-1$ in $R$, and by minimality of $g$, we must have this same element appear in $S$.  Eventually, though, we encounter a contradiction in the first row:  we must have selected some entry greater than $2$ (and less than $e$) for the first row of $R$, yet this entry cannot appear in $S$.  It certainly can't appear in the first row since we selected $2$ already, nor can it appear in any subsequent row of $S$ by the choices we already assumed (in order to avoid contradicting minimality of $g$) earlier.
\end{proof}

Since the rows with sizes larger than $j$ must be zero, the non-zero part of the boundary block of $\widetilde\P_{i,j}(n)$ for $j < i$ is square, with dimensions $\mu_{j-1}(n-1) \times \mu_{j-1}(n-1)$ by Lemma~\ref{l:sso_biject}.  Otherwise, when $i \leq j$, we have $\mu_{i-1}(n-1)$ rows.

Referring back to the diagram calculus in Example~\ref{e:phi}, the number of columns in the boundary block of $\widetilde\P_{i,j}(n)$ whose indexing set contains a $2$-entry is the same as the number of $(\leq (j-1))$ super-standard subsets of $(n-1)$ lying in the interior, or equivalently, the number of $(\leq (j-1))$ super-standard subsets of $(n-2)$, which is $\mu_{j-1}(n-2)$.

So although the matrices do not necessarily agree on every entry, they do agree by Lemma~\ref{l:ssoequiv} on $\mu_{j-1}(n-2)$ columns.  Presumably, we could diagonalize the matrices on the columns where they agree and use this to clear all the other entries in order to obtain the result, but we have not verified this in detail.  We believe it is plausible that under appropriate parameter choices $\widetilde\P_{i,j}(n)$ is integrally equivalent to a block matrix of the form
\newcommand\Tstrut{\rule{0pt}{2.6ex}}         
\newcommand\Bstrut{\rule[-0.9ex]{0pt}{0pt}}  
\begin{align*}
\widetilde{\P}_{i,j}(n) &\sim \left [ \begin{array}{c|c}
\widetilde{\P}_{i-1,j-1}(n-1) & 0 \Tstrut\Bstrut \\
\hline
\ast & \widetilde{\P}_{i,j}(n-1) \Tstrut\Bstrut 
\end{array} \right ]
\end{align*}
and that a proof by induction should then go through for Conjecture~\ref{c:sso}.

\bigskip
\section*{Acknowledgements}

This project began at the 2023 Research Experiences for Undergraduates (REU)
site at James Madison University, supported by the National Science
Foundation through DMS grant 1950370.  


\newcommand{\etalchar}[1]{$^{#1}$}
\providecommand{\bysame}{\leavevmode\hbox to3em{\hrulefill}\thinspace}
\providecommand{\MR}{\relax\ifhmode\unskip\space\fi MR }
\providecommand{\MRhref}[2]{%
  \href{http://www.ams.org/mathscinet-getitem?mr=#1}{#2}
}
\providecommand{\href}[2]{#2}

\end{document}